\documentclass[12pt]{article}

\usepackage[english]{babel}

\usepackage{amsmath, amssymb, amsthm, amsfonts, dsfont, bbm}
\usepackage{subcaption} 

\newtheorem{definition}{Definition}
\newtheorem{lem}{Lemma}
\newtheorem{thm}{Theorem}
\newtheorem{prop}{Proposition}

\newtheorem{claim}{Claim}
\newtheorem{cor}{Corollary}
\newtheorem{rk}{Remark}

\usepackage{graphicx, xcolor}

\usepackage[style=alphabetic, sorting=nyt]{biblatex}
\AtBeginBibliography{\raggedright\sloppy}
\AtEveryBibitem{%
  \clearfield{note}%
   \clearfield{doi}%
}

\usepackage[colorlinks=true, allcolors=blue]{hyperref}

\usepackage[margin=1in]{geometry}  

\usepackage{comment, etoolbox, ulem} 

\newcommand{\Z}{\mathbb{Z}}

\newcommand{\modif}[1]{\textcolor{magenta}{#1}}

\providecommand{\keywords}[1]{\textbf{Keywords:} #1}

\title{The Rightmost Particle of the Contact Process on Dynamic Random Environments}
\author{Isabella Alvarenga\footnote{ \textsc{Department of Statistics, University of Warwick, United Kingdom}\texttt{isabella.goncalves-de-alvarenga@warwick.ac.uk}}, Aurelia Deshayes\footnote{ \textsc{Univ Paris Est Creteil, Univ Gustave Eiffel, CNRS, LAMA UMR8050, F-94010 Creteil, France} and \textsc{IRL CNRS IFUMI-2030, Montevideo, Uruguay}  \texttt{aurelia.deshayes@u-pec.fr}}}

\begin{document}

\maketitle

\begin{abstract}
We study the behaviour of the rightmost occupied site in two models: the Spont process and the contact process with inherited sterility, in dimension 1. Both can be viewed as contact processes evolving in dynamic random environments, where the environment may itself depend on the state of the process. In the Spont process, blocking particles appear spontaneously, while in the inherited sterility model, sterile sites arise as offspring of occupied ones. Each model presents distinct mathematical challenges: the Spont process lacks self-duality, whereas the inherited sterility process is non-attractive. We establish a law of large numbers and a central limit theorem for the position of the rightmost occupied site. Our approach is based on the construction of a sequence of renewal times, defined through a detailed analysis of active infection paths. These results are obtained in the supercritical regime of the Spont process.
\end{abstract}




\section{Introduction}\label{Intro}

\paragraph{}The processes studied in this work are both modifications of a well-known process called the \textbf{contact process} introduced in \cite{harris1974contact}. This classical contact process is a continuous-time Markov process~$(\zeta_t)_{t\geq 0}$ taking values in~$\{0,1\}^{\mathbb{Z}}$ where sites (elements of~$\mathbb Z$) in state $1$ are said to be \textbf{occupied} and sites at state $0$ are called \textbf{empty}. This process depends on only one parameter~$\lambda$ and is described as follows: occupied sites die at rate $1$ leaving that site empty, and empty sites are occupied at rate~$\lambda$ times the number of nearest occupied neighbours. One can think of this process as a model for population dynamics, where each occupied site corresponds to an individual that can reproduce or die.

Regarding this classical contact process, we highlight two properties: the first one is the existence of a phase transition where, in one regime, the whole population certainly goes extinct, whereas in the other it has a positive chance of survival; the second result states that, if the population survives, it spreads linearly in space. We make this notion precise in the next paragraphs.

Consider a contact process~$(\zeta_t^0)_{t \geq 0}$ with parameter~$\lambda$, started from a single occupied site at the origin. There exists a critical value~$\lambda_c \in (0, \infty)$ such that the process dies out—that is, reaches the configuration where every site is empty—with probability one if and only if~$\lambda \leq \lambda_c$:
\[
\mathbb{P}\left( \exists t \geq 0 : \forall x \in \mathbb{Z},~\zeta_t^0(x) = 0\  \right) = 1 \quad \text{if and only if} \quad \lambda \leq \lambda_c.
\]

Note that the configuration in which all sites are empty is an absorbing state for the process. When~$\lambda > \lambda_c$, the process is said to be supercritical; in that case, the process has a positive probability of what we then call survival. If, on the other hand~$\lambda \leq \lambda_c$, we say that the process is subcritical. The existence of this phase transition was established in~\cite{harris1974contact}, with the critical case~$\lambda = \lambda_c$ being treated in~\cite{durrett_griffeath1983supercritical}. 

For a supercritical contact process~$(\zeta_t^0)_{t\geq 0}$, there exists~$\alpha = \alpha (\lambda)>0$ such that, for the law of this process conditioned on survival, one has the following:    
    \begin{equation} \label{speedCP}
        \lim_{t\rightarrow\infty}\frac{r(\zeta_t^0)}{t}=\alpha \text{ almost surely and in }L^1
    \end{equation}
where~$r(\zeta_t^0)=\sup\{x\in\mathbb Z\,:\,\zeta_t^0(x)=1\}$ denotes the position of the rightmost occupied site at time $t$. The proof of~\eqref{speedCP} in the $L^1$ sense is given in~\cite{durrett1980growth}, while the almost sure version appears in~\cite{durrett_griffeath1983supercritical}. 

The contact process has a very convenient property of being attractive, which informally means that the more occupied sites we have in the initial configuration, the more occupied sites we will have at all times. This property plays a key role in the proof of~\eqref{speedCP} allowing the use of subadditive techniques (\cite{kingman1973subadditive,liggett1985interacting}). A little later, Galves and Presutti~\cite{galvespresutti} on one side and Kuczek~\cite{kuczek1989central} on the other proved a central limit theorem for the trajectory of the rightmost particle. Kuczek in particular constructs break points that allow him to divide the trajectory into i.i.d. pieces; we will come back to this later. 

The contact process model is very rich and has been the subject of numerous extensions (\cite{krone1999,durrett_schinazi2000boundary,broman2007stochastic,garetmarchand2014bacteria,deshayes2014}). In this work, we study two variants of the contact process introduced by~\cite{velasco2024}: the \textbf{Spont process} and the \textbf{contact process with Inherited Sterility}. Both can be seen as contact processes on dynamic random environments, which have attracted increasing attention in recent years. 

The first work on that was done in \cite{broman2007stochastic}, where the author introduced a model for which the recovery rate at each site changes over time (this is the environment); this model was further developed in \cite{steif2007critical} and \cite{remenik2008contact} where it is proved that the critical value is independent from the initial environment and that the process dies out at criticality (using a Bezuidenhout-Grimmett type construction). In~\cite{linker2020contact} the authors considered a dynamic bond percolation, that is, edges between neighbouring lattice sites open and close for transmission, and studied the behaviour of the critical value with respect to the parameters (the velocity of the changes and the density of the open edges); some questions left open have been completed by~\cite{hilario2022results}.  This type of process has also been the subject of recent studies on other graphs (Bienaymé–Galton–Watson trees in~\cite{cardona2024contact}, random d-regular graphs with switching bond dynamics in~\cite{leite2024contact} and~\cite{schapira2023contact}).


In most of these models, there are two processes involved: a Markovian background which represents the environment and which \textit{does not} depends on the contact process, and the contact process itself which evolves according to the background process. An important difference for the models we are studying here is that the background and the contact process are \textit{intertwined}, or, in other words, that the environment \textit{depends} on the contact process. 

The two variants studied in the present work, abbreviated by Spont and IS processes, have an additional state in addition to the classic empty ($0$) and occupied ($1$). This new state $-1$ is considered \textbf{blocked} (or sterile for IS) in the sense that it prevents the propagation of the population of $1$. In the Spont process, blocked sites appear spontaneously. In the IS process, they appear as descendants of occupied ones; we will sometimes refer to 1s as fertile individuals. For parameters $\lambda\in [0,\infty)$ and $p\in [0,1]$, the dynamics are described as follows:
\begin{itemize}
    \item \textbf{Spont:}  non-empty sites, whether of type $1$ or $-1$, become empty at a rate $1$; empty sites become spontaneously blocked at rate~$2\lambda\times (1-p)$; finally, occupied sites send a copy of themselves to each adjacent empty site at a rate of $\lambda \times p$. 
 
 \item \textbf{IS:} non-empty sites, whether of type $1$ or $-1$, become empty at a rate of $1$; sites occupied by 1 can be thought as fertile individuals that produce offspring, sending a child to each adjacent empty site at a rate of $\lambda$. Then, with probability $p$, this descendant is also a fertile individual and thus will have state $1$; however, with probability $1-p$, that descendant will be blocked (or sterile) and the site will then have state $-1$.
\end{itemize}

These two variants of the contact process were introduced in~\cite{velasco2024}. In that work, Velasco introduced the contact process with inherited sterility as a framework for studying biological control strategies aimed at suppressing invasive species through the release of modified individuals. This approach extends the Sterile Insect Technique (SIT)~\cite{SIT}, developed in the 1950s by E. Knipling to combat the New World screw worm. SIT relies on the release of radiation-sterilized males so that matings with fertile females do not produce offspring, leading to population decline. However, complete sterilization often requires radiation levels that impair male competitiveness. In the inherited sterility method~\cite{ISI_north}, lower radiation doses produce partially sterile offspring, which can themselves generate both fertile and sterile descendants. The transmission of sterility across generations compensates for reduced mating success while still driving population suppression. The Spont process was introduced as a tool to study the IS process, but it also has an independent interest as a model in its own right, and it can be seen as a particular case of the continuous version of the bacteria model introduced in~\cite{garetmarchand2014bacteria}.

In this work, we prove a law of large numbers (Theorems~\ref{thm_Spont_Speed} and~\ref{thm_Speed_IS}) and a central limit theorem (Theorems~\ref{thm_Spont_CLT} and~\ref{thm_CLT_IS}) for the position of the rightmost site in state~$1$ in both the Spont process and the contact process with inherited sterility. Our results hold only for the set of parameters for which the Spont process can survive, which is potentially smaller than the set of parameters for which the IS process survives. Even for this, and even though we are in dimension 1, we stress that proving the ballistic motion of the rightmost particle is not easy. 

The main challenge in studying the contact process with inherited sterility is the lack of attractiveness in the model. Indeed, an increase in the number of fertile individuals can also lead to a rise in sterile individuals, paradoxically reducing the process's chances of survival and spread. There are other models in which linearity has been proven without attractivity on dimension~1. This is the case, for example, of the East model~\cite{blondel,GLM15} and the Fredrickson-Andersen one spin facilitated model (FA-1f)~\cite{blondeldeshayes}. These are both kinetically constrained models for which we know that Bernoulli is an invariant measure and for which there exist strong results of relaxation to equilibrium \cite{BCMRT13}, which is a key element in studying the behavior of the front.

The main idea of our work is to identify, for each process, a sequence of special times—called renewal times—at which a specific property holds. These times are chosen so that the increments of the rightmost particle’s position between them are i.i.d. and occur often enough to prevent large deviations in between. The main difficulty is that renewal times are not stopping times, as they depend on the entire future of the process. This renewal-time approach is classical in the study of random walks in random environments (see, e.g., \cite{sznitman2004topics}) and has also been applied to the contact process, notably by \cite{kuczek1989central}. Variants of this method appear in later works such as \cite{mountford2016functional} and \cite{mountford2019asymmetric}.

This work is structured as follows. In Section~\ref{Sec_defAndResults}, we define both models and state the main results. In Section~\ref{Sec_SpectialProperty} and~\ref{Sec_Prop}, we define the special property mentioned above, and prove his occurrence. In Section~\ref{Sec_Markov} we show that the displacements of the rightmost particle are i.i.d. under an appropriate conditional probability measure, which leads to the proof of the main theorems in Section~\ref{sec_ProofMainTheorems}. Finally, Section ~\ref{appendix_construction} is devoted to construct rigorously the processes through graphical paths.

\section{Models and main results} \label{Sec_defAndResults}

\subsection{Notations about configurations} 

\paragraph{} Both processes we study are continuous-time Markov processes taking values on the space of configurations $\{-1,0,1\}^{\mathbb{Z}}$. We recall that a site in state $1$ is said to be \textbf{occupied} (by a fertile particle), a site in state $-1$ is said to be \textbf{blocked} (by a sterile particle) and a site in state $0$ is called \textbf{empty}. Both of these models are parametrized by two parameters $\lambda\in [0,\infty)$ and $p\in [0,1]$.

For $x\in \mathbb{Z}$, we consider    
    \begin{equation*}
       \mathcal{C}^x=\left\{c\in \{-1,0,1\}^{\mathbb{Z}}\,:\,\sup\{y\in \mathbb{Z}\,:\,c(y)=1\}=x\right\} 
    \end{equation*}
    the subspace of configurations with a rightmost 1 at site $x$. 
    Let also $\mathcal{C} = \cup_{x\in \mathbb{Z}}\mathcal{C}^x$. 

For $c\in \mathcal{C}$, we can define its rightmost occupied particle:
\[r(c)=\sup\{x\in \mathbb{Z}\,:\,c(x)=1\}.\]

For $x\in \mathbb{Z}$, let $\delta_x^{-}$ be the configuration where $x$ is occupied all the other are blocked sites. We often refer to the configuration $\delta_x^-$ as the most \textit{hostile} environment. Formally, let: 

        \begin{equation*}
        \delta_x^{-}(y)=
            \begin{cases}
                -1 &\text{ if }y\neq x \\
                ~1 &\text{ if }y=x.
            \end{cases}
        \end{equation*}
For two configurations $c_1,c_2$ in $\{-1,0,1\}^{\mathbb{Z}}$, we say that $c_1 \lesssim c_2$ if $c_1(x)\leq c_2(x)$ for all $x\in \mathbb{Z}$.

\subsection{Spont dynamics}

\paragraph{} For the Spont process, we denote the state of a site $x\in \mathbb{Z}$ at time $t\in[0,\infty)$ by $\xi_t(x) \in \{-1,0,1\}$. The transition rules for a site $x\in \mathbb{Z}$ and a given configuration $\xi\in \{-1,0,1\}^{\mathbb{Z}}$ are as follows:

\begin{eqnarray*}
    0 &\longrightarrow 1 \, &\text{ at rate } \lambda \times p \times |\{y\sim x\}\,:\,\xi(y)=1|, \\
     0 &\longrightarrow -1\,&\text{ at rate } 2\lambda \times (1-p), \\
    1,-1 &\longrightarrow 0\, &\text{ at rate }1.
\end{eqnarray*}
where we write~$x\sim y$ for when~$x,y\in\mathbb Z$ are such that~$|x-y|=1$. 

We emphasize a few remarks regarding this model. The first one is that, similar to the model in~\cite{remenik2008contact}, occupied sites cannot give birth on top of blocked sites and blocked sites do not reproduce. However, unlike the model of~\cite{remenik2008contact}, blocked sites \textit{cannot} appear on top of occupied sites. Note also that when~$p=1$, this process reduces to the classical contact process. So far, the parametrization of the Spont process in terms of $\lambda$ and $p$ might have come across as something mysterious, as one could have actually introduced an independent parameter to rule the blocking phenomena but it will make sense in Remark~\ref{CouplingMainThm}.

We denote by $(\xi_t^c)$ the process starting with initial configuration $c\in\mathcal C$ (it well be constructed in subsection~\ref{section_graph_constr}). For this model, the subspace of configurations without occupied sites is absorbing in the sense that no occupied sites can no longer appear. If the process reaches such a configuration, we say that it dies out and we can define its extinction time $\tau_{\xi}$ as the first moment when there are no occupied sites present in the process:
\begin{equation}
     \tau_{\xi}({c}) = \inf\{t\geq 0\,:\,\forall x\in \mathbb{Z},\,\xi_t^c(x)\neq 1\}. 
\end{equation}

By comparing the Spont process with a well-chosen oriented percolation, Velasco~\cite{velasco2024} proved a phase transition phenomenon. 

\begin{rk}[Theorem 4 of~\cite{velasco2024}]\label{SupercriticalSpont}
    For $\lambda>\lambda_c$, there exists a $\tilde{p}\in[\frac{\lambda_c}{\lambda},1)$ such that the process $(\xi_t^c)_{t\geq 0}$ started from any $c\in\mathcal C$ with parameters $\lambda>\lambda_c$ and $p>\tilde{p}$ survives with positive probability, i.e., $\mathbb{P}(\tau_\xi(c)=\infty)>0$. 
\end{rk}

As we have mentioned before, we are interested in the behaviour of the rightmost occupied site over time. Regarding this observable, we will prove both a strong law of large numbers and a central limit theorem. These are formally stated in the next theorems.


\begin{thm}[Speed of the rightmost particle for Spont] \label{thm_Spont_Speed}
    Let $(\xi_t)$ be a Spont process with parameters~$\lambda>\lambda_c$ and~$p>\tilde{p}$ as in Remark~\ref{SupercriticalSpont}. Then, there exists a constant~$\mu_{\mathrm{SP}}>0$ such that, for any initial configuration $c\in\mathcal C$, conditioned on survival, we have that
    \begin{equation}\label{eq_speed_Spont}
        \frac{r(\xi_t^c)}{t}\xrightarrow[t\rightarrow\infty]{} \mu_{\mathrm{SP}} \quad \text{ almost surely. }
    \end{equation}
\end{thm}

\begin{thm}[CLT the rightmost particle for Spont] \label{thm_Spont_CLT}
  Let $(\xi_t)$ be a Spont process with parameters~$\lambda>\lambda_c$ and~$p>\tilde{p}$ as in Remark~\ref{SupercriticalSpont}. Then, there exists a constant~$\sigma_{\mathrm{SP}}>0$ such that, for~$\mu_{\mathrm{SP}}$ given as in~\eqref{eq_speed_Spont}, for any initial configuration $c\in\mathcal C$, conditioned on survival, we have that
    \begin{equation}
        \frac{r(\xi_t^c) - \mu_{\mathrm{SP}} t}{t^{1/2}}\xrightarrow[t\rightarrow\infty]{\mathrm{(dist)}}\mathcal{N}(0,\sigma_{\mathrm{SP}}^2).
    \end{equation}
\end{thm}

\subsection{IS dynamics}

\paragraph{} For the contact process with inherited sterility, we denote the state of a site~$x\in \mathbb{Z}$ at time~$t\in[0,\infty)$ by $\eta_t(x) \in \{-1,0,1\}$. The transition rules for a site $x\in \mathbb{Z}$ and a given configuration $\eta\in \{-1,0,1\}^{\mathbb{Z}}$ are as follows:
\begin{eqnarray*}
    0 &\longrightarrow 1\,&\, \text{ at rate } \lambda \times p \times |\{y\sim x\}\,:\,\eta(y)=1|, \\
     0 &\longrightarrow -1\,&\,\text{ at rate } \lambda \times (1-p) \times |\{y\sim x\}\,:\,\eta(y)=1|,\\
     1,-1 &\longrightarrow 0\,& \text{ at rate }1.
\end{eqnarray*}

We emphasize some important aspects of this dynamics: sterile individuals \textit{do not} reproduce and fertile individuals are \textit{not allowed} to send descendants on top of them. Because of that, one can think that sterile individuals act like blocking sites where the fertile individuals cannot reproduce. 


Once again, we denote by $(\eta_t^c)$ the process starting from $c\in\mathcal C$ and we note that the subspace of configurations without any fertile individuals is absorbing under this dynamics. Therefore, as we have done for the Spont process, we define $\tau_{\eta}$ the extinction time of an IS process as:
\begin{equation}
     \tau_{\eta}({c}) = \inf\{t\geq 0\,:\,\forall x\in \mathbb{Z},\,\eta_t^c(x)\neq 1\}. 
\end{equation}

In the subsection~\ref{section_graph_constr}, we will prove the following result.

\begin{rk} \textbf{(Coupling)} \label{CouplingMainThm}
    Consider three configurations~$\xi_0\lesssim \eta_0\lesssim \zeta_0$ such that~$\xi_0,\eta_0\in\{-1,0,1\}^{\mathbb{Z}}$ and~$\zeta_0\in\{0,1\}^{\mathbb{Z}}$. Let~$\lambda\in[0,\infty)$ and $p\in [0,1]$ be given. Consider~ $(\xi_t)_{t\geq 0}$ a Spont process started from~$\xi_0$ and~$(\eta_t)_{t\geq 0}$ a contact process with inherited sterility started from~$\eta_0$, both with parameters~$\lambda$ and~$p$, and~$(\zeta_t)_{t\geq 0}$ be a contact process started from~$\zeta_0$ with parameter~$\lambda\times p$. Then, there exists a coupling of these three processes such that the following condition is satisfied:  
    \begin{equation*}
        \xi_t \lesssim \eta_t \lesssim \zeta_t,\text{ a.s. for all }t\geq 0.
    \end{equation*}
\end{rk}

Figure~\ref{fig_Coupling} illustrates the coupling from Remark \ref{CouplingMainThm}: sites occupied by the three processes are shown in white, the ones occupied by IS and contact (but not Spont) are in light gray and the ones occupied only by the contact process are in dark gray. Empty sites are shown in black. For clarity, blocked sites do not appear, but their presence can sometimes be guessed from vertical blockages. From remarks \ref{SupercriticalSpont} and \ref{CouplingMainThm} one can obtain the following result.

\begin{figure}[h]
    \centering
    \includegraphics[width=0.5\linewidth]{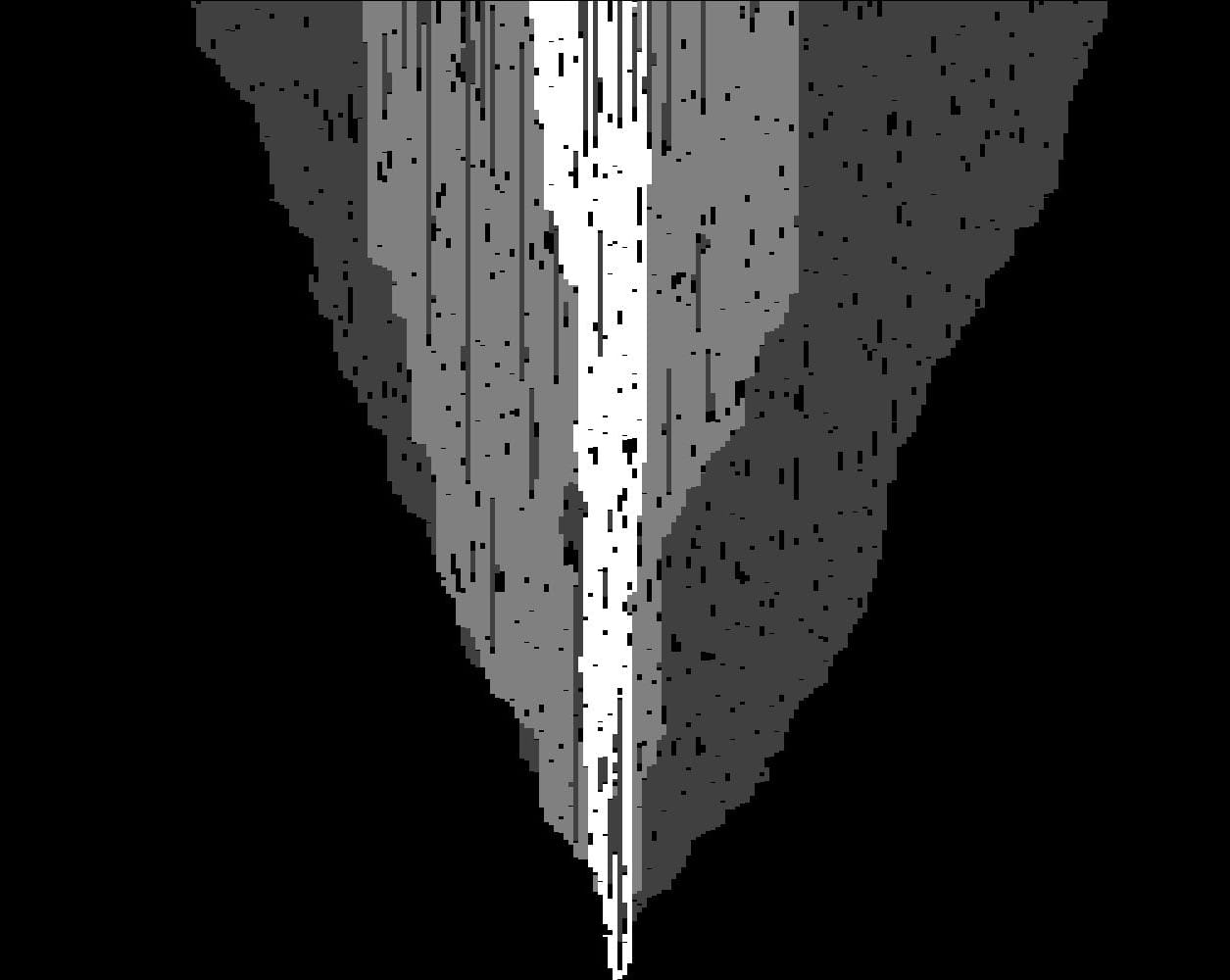}
    \caption{Simulation of the coupled processes.}
    \label{fig_Coupling}
\end{figure}

\begin{rk}[Theorem 1 of~\cite{velasco2024}] \label{ISPosChanceSurvival}
    For~$\lambda>\lambda_c$, there exists~$\bar{p}\in [\frac{\lambda_c}{\lambda},1)$ such that the inherited sterility process~$(\eta_t^c)_{t\geq 0}$ started from any
    ~$c\in\mathcal C$ with parameters~$\lambda$ and~$p>\bar{p}$ survives with positive probability, i.e.,~$\mathbb{P}(\tau_\eta(c)=\infty)>0$. 
\end{rk}

This suggests that, a priori, there could exist a regime in which the inherited sterility process survives while the Spont process dies out. However, we do not explore this regime in the present work. Similar to the Spont process, we denote by~$\mathcal{S}_\eta(c)$ the survival event of an IS process started from~$c\in\mathcal C$, i.e.,
\[
\mathcal{S}_\eta(c)=\{\tau_\eta(c)=\infty\}
=\{\forall t\geq 0, \exists y\in\Z : \eta_t^c(y)=1\}.\]

In the following, we will denote by 
\[
\mathcal{S}_\xi(c)=\{\tau_\xi(c)=\infty\}
=\{\forall t\geq 0, \exists y \in \Z: \xi_t^c(y)=1\}
\]
the survival event of a Spont process started from~$c\in\mathcal C$.

We finally state our main result regarding the contact process with inherited sterility. Similar to what we have obtained for the Spont process, we have a strong law of large numbers and also a central limit theorem for the position of the rightmost fertile individual.

\begin{thm}[Speed of the rightmost particle for IS] \label{thm_Speed_IS}
Let~$(\eta_t)_{t\geq 0}$ be an IS process with parameters~$\lambda>\lambda_c$ and~$p>\tilde{p}$ as in Remark~\ref{SupercriticalSpont}.
      Then, there exists~$\mu_{\mathrm{IS}}>0$ such that for any inital configuration $c\in\mathcal C$, conditioned on $S_\eta(c)$, we have that
     \begin{equation}\label{eq_speed_IS}
        \frac{r(\eta_t^c)}{t}\xrightarrow[t\rightarrow\infty]{} \mu_{\mathrm{IS}} \quad \text{ almost surely. }
    \end{equation}
\end{thm}

\begin{thm}[CLT for the rightmost particle for IS]\label{thm_CLT_IS}
    Let~$(\eta_t)_{t\geq 0}$ be an IS process with parameters~$\lambda>\lambda_c$ and~$p>\tilde{p}$ as in Remark~\ref{SupercriticalSpont}. Then, there exist~$\sigma_{\mathrm{IS}}\geq0$ such that for~$\mu_{\mathrm{IS}}\geq 0$ given as in~\eqref{eq_speed_IS}, for any initial configuration $c\in\mathcal C$, conditioned on survival, we have that

    \begin{equation}
        \frac{r(\eta_t^c) - \mu_{\mathrm{IS}} t}{t^{1/2}}\xrightarrow[t\rightarrow\infty]{\mathrm{(dist)}}\mathcal{N}(0,\sigma_{\mathrm{IS}}^2).
    \end{equation}
\end{thm}

\subsection{Graphical construction}\label{section_graph_constr}

\paragraph{} Since the argument we will use to prove the main theorems strongly rely on the graphical construction of those processes, we devote more than a few words for that. For a given choice of parameters $\lambda \in [0, \infty)$ and $p\in [0,1]$, we consider a graphical construction (also known as a Harris construction) $\mathcal{H}$ given by:

\begin{equation} \label{FormOfGraphicalConstruction}
    \mathcal{H} = \Big((N_1^{x,y})_{x\sim y}, (N_2^{x,y})_{x\sim y}, (U^x)\Big)
\end{equation}
where for $x,y\in \mathbb Z$ we have that $N_1^{x,y}$, $N_2^{x,y}$ and $U^x$ are Poisson point processes on $[0,\infty)$ with rates $\lambda p$, $\lambda (1-p)$ and $1$, respectively, and they are all independent from each other. 

Consider the space~$\mathbb{Z} \times [0, \infty)$, which can be thought of as attaching a time axis to each site of~$\mathbb{Z}$. At each arrival~$t\in N_1^{x,y}$, draw an arrow~$\overset{1}{\rightarrow}$ from~$(x,t)$ to~$(y,t)$; at each arrival~$t\in N_2^{x,y}$, draw an arrow~$\overset{-1}{\rightarrow}$ from~$(x,t)$ to~$(y,t)$ and a~$\triangle$ at $(y,t)$; finally, at each arrival~$t\in U^x$, draw a~$\times$ at the point~$(x,t)$. An illustration of such constriction is done in Figure~\ref{GCSPONTCPIS}.  

\begin{figure}[h]
    \centering
    \includegraphics[width=0.9\linewidth]{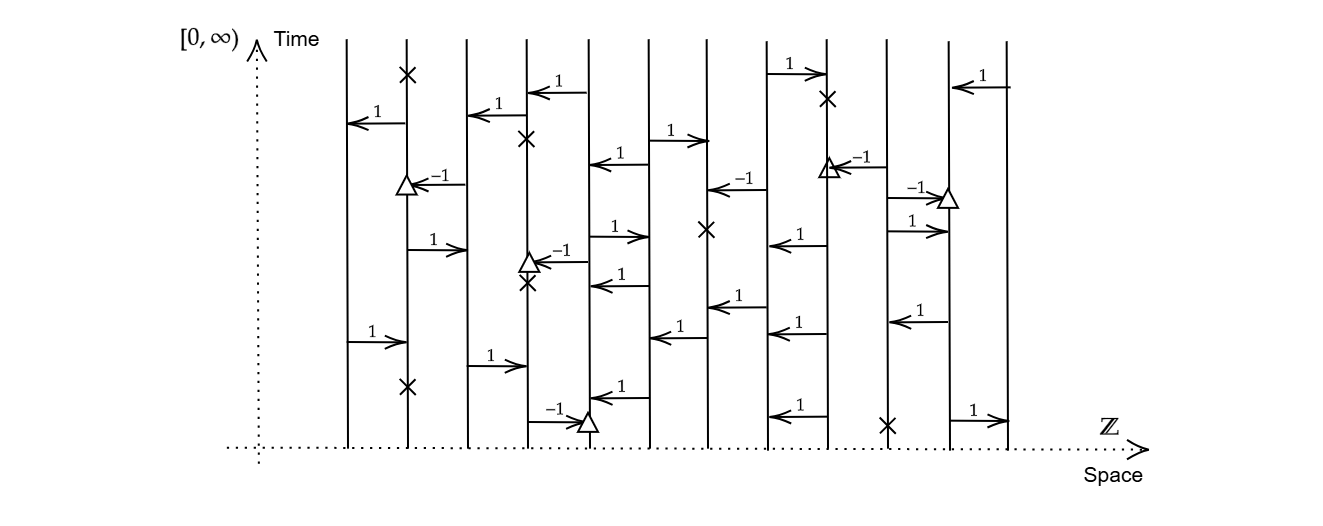}
    \caption{Illustration of graphical construction}
    \label{GCSPONTCPIS}
\end{figure}

The space-time marks left by  the processes $N_1^{x,y}$ and $N_2^{x,y}$ are called \textbf{transmission arrows} of type $1$ and type $-1$, respectively. The mark $\triangle$ where a transmission arrow of type $-1$ points towards is called \textbf{blocking mark}. Finally, a realisation of $U^x$ is referred as a \textbf{healing mark}. 

Starting from configurations $\xi_0,\eta_0\in\{-1,0,1\}^\Z$ and $\zeta_0\in\{0,1\}^\Z$, we construct a Spont process~$(\xi_t)_{t\geq 0}$, an IS process~$(\eta_t)_{t\geq 0}$ and a contact process~$(\zeta_t)_{t\geq 0}$ using the same collection~$\mathcal{H}$:
\begin{itemize}
    \item At each arrow $\overset{1}{\rightarrow}$ from~$(x,t)$ to~$(y,t)$, the three processes may be updated at $y$ according to a condition on the status of $x$ and $y$: if $x$ is in state 1 and $y$ in state 0, the birth of a fertile individual occurs in $y$, that is, its state flips to 1.
\item  At each arrow $\overset{-1}{\rightarrow}$ from~$(x,t)$ to~$(y,t)$, Spont and IS processes may be updated at $y$ according to a condition on the status of $x$ and/or $y$ (these arrows are ignored by contact process):  if $y$ is in state $0$, a blocked individual appears spontaneously in $y$ for the Spont process, that is, its state flips to -1; for IS process it requires moreover that $x$ is in state 1 and it can be interpreted like the birth of a sterile individual.
\item At each healing mark $(x,t)$, the three processes update the state of $x$ to $0$.
\end{itemize}

We call this construction of $(\xi_t,\eta_t,\zeta_t)$ the basic coupling started from $(\xi_0,\eta_0,\zeta_0)$. We denote by $\mathbb P,\mathbb E$ the associated probability and expectation. This probability space allows us to construct the three processes with any initial configuration simultaneously. This basic coupling satisfies the Remark~\ref{CouplingMainThm} (one can check that every transition preserves the order, it is done in detail in~\cite{velasco2024}). 

For~$s\in[0,\infty)$, we let~$\mathcal H[s,\infty)$ be the graphical construction on~$[s,\infty)$ obtained by the restriction of~$\mathcal H$ by~$s$ units of time in the future; i.e., ~$\mathcal H[s,\infty)$ is given by the collection~$\left((N_1^{x,y}\cap [s,\infty))_{x\sim y}, (N_2^{x,y}\cap [s,\infty))_{x\sim y}, (U^x\cap [s,\infty))\right)$ for when~$\mathcal H$ is given as in~\eqref{FormOfGraphicalConstruction}. Thus, we can define~$(\chi_{s,t}^d)_{t\geq s}$ the process started from~$d\in\{-1,0,1\}^{\mathbb{Z}}$ at time~$s$ constructed with~$\mathcal H[s,\infty)$.



\begin{definition}
    \textbf{(Active infection path)} Let~$(\chi_t)_{t\geq 0}$ denote either a Spont process or an inherited sterility process started from~$\chi_0$, constructed using a graphical construction~$\mathcal H$ as in\eqref{FormOfGraphicalConstruction}. Let~$I\subset [0,\infty)$ be a time interval. We say that $\Gamma:I\rightarrow \mathbb{Z}$ a càdlàg function is an active infection path for~$(\chi_t)_{t\geq 0}$ if $\Gamma$ satisfies the following three properties:

    \begin{itemize}
        \item \textbf{(P1)}: for all $s\in I$, $s\notin U^{\Gamma(s)}$ for any $x\in \mathbb{Z}$
        \item \textbf{(P2)}: if $\Gamma(s-)\neq \Gamma (s)$, then $\Gamma(s)\in N^{\Gamma(s-), \Gamma(s)}_1$ for all $s\in I$
        \item \textbf{(P3)}: $\chi_s(\Gamma(s))\neq -1$ for any $s\in I$ 
    \end{itemize}

     If there exists such $\Gamma:I\rightarrow \mathbb{Z}$ with $x=\gamma(s)$ and $y=\gamma(t)$ for $s\leq t$ where $[s,t]\subset I$, we write either~$(x,s)\Longrightarrow(y,t)$ or~$(x,s)\overset{\mathcal{H},\chi_0}{\Longrightarrow} (y,t)$ in case we want to highlight the graphical construction and the initial configuration used. 
\end{definition}

A natural but problematic approach to constructing the processes would be to declare a site~$x$ occupied at time~$t$ if and only if there exists an active infection path from an initially occupied site (at time~$0$) reaching~$x$ at time~$t$. However, this leads to a circular definition, since condition~(P3) depends on the process itself. Fortunately, the constructions provided in Section~\ref{appendix_construction} for $(\chi_t)_{t \geq 0}$ being a Spont or an inherited sterility process avoid this issue and prove the following result.

\begin{lem}\label{ActivePathIFF}
    Let $(\chi_t)_{t\geq 0}$ be a Spont or IS process started from $\chi_0$ constructed with a graphical construction $\mathcal{H}$ . Let $(x,t)\in\mathbb{Z}\times [0,\infty)$. Then, $\chi_t(x)=1$ if and only if there exists an active infection path $(y,0)\overset{\mathcal H, \chi_0}{\Longrightarrow} (x,t)$ for some $y\in\mathbb{Z}$ with $\chi_0(y)=1$. 
\end{lem}

\subsection{Attractivity and lack of attractivity}

\paragraph{}Thanks to the previous construction, we have that the classical contact process and the Spont process are attractive, but not the contact process with inherited sterility. More precisely:

\begin{rk}
    For initial configurations~$c_1,c_2\in\{0,1\}^{\mathbb Z}$ and~$c_1',c_2'\in\{-1,0,1\}^{\mathbb Z}$, the processes heve been coupled in such way that
    \begin{align*}
        c_1\lesssim c_2 \Rightarrow \forall t \geq 0,~\zeta_t^{c_1}\lesssim \zeta_t^{c_2},\\
     c_1'\lesssim c_2' \Rightarrow \forall t \geq 0,~\xi_t^{c_1'}\lesssim \xi_t^{c_2'}.
        \end{align*}
\end{rk}
This property does not hold for the inherited sterility contact process. Moreover, the contact process is also additive, that is for all $t \geq 0$ we have that~$\zeta_t^{c_1\vee c_2}=\zeta_t^{c_1}\vee \zeta^{c_2}_t$ where~$c_1\vee c_2$ is the configuration defined as~$(c_1\vee c_2)(x)=\max\{c_1(x),c_2(x)\}$ . This property is not satisfied by either the Spont process or the IS process. 


Now that the main objects we will consider are well defined and the fundamental properties are stated, we take a moment to explain why certain standard techniques fail in our setting. A natural first attempt to prove a strong law of large numbers for the position of the rightmost~$1$ would be to apply the subadditive ergodic theorem, as is done for the classical contact process. In \cite{liggett1985interacting}, Liggett introduces a “filling-up” procedure at time~$m$ by occupying all sites to the left of the rightmost particle at that moment. However, because the IS process lacks attractiveness, this construction does not yield the required subadditive property. More generally, the absence of attractiveness creates an obstruction that persists regardless of how one attempts to modify the configuration at time~$m$: even if one devised a different procedure than “filling up,” the argument would break down at essentially the same point. Furthermore, in the case of attractive models, this technique only gives the existence of a speed (which could be 0) and provides no information about fluctuations. 

One could also attempt to adapt the approach of \cite{kuczek1989central}, where for the classical contact process one defines renewal times corresponding to moments when the process started from a single occupied site located at the position of the rightmost particle survives indefinitely. In that setting, renewal times have the crucial property that the rightmost particle at all future times is a descendant of the rightmost particle at the renewal time. This allows one to connect rightmost particles across renewals via infection paths, which in turn provides control over their asymptotic behaviour. For the Spont and IS processes, however, this strategy fails for a subtler reason related to the lack of additivity. Suppose we start from the configuration with only the rightmost site occupied, but alter the surrounding landscape by introducing blocking sites at every other location (or any similar modification that preserves the rightmost particle at its original position). In such cases, it may happen that the rightmost particle in the modified process does not coincide with the rightmost particle in the original process—or worse, that survival in the modified system does not imply survival in the original one. Thus, in contrast to the classical case, we must impose a stronger condition to ensure a workable renewal structure, as will be clarified in the next section.

\section{Special property and renewal times} \label{Sec_SpectialProperty}

\paragraph{} In this section, we examine a structural property exhibited by both the Spont process and the contact process with inherited sterility. In general terms, this property reflects a kind of consistency over time: given two times~$t_1, t_2 \in [0, \infty)$ with~$t_1 \leq t_2$, if the property holds throughout the interval ~$[t_1, t_2]$ and continues to hold from time~$t_2$ onward, then we can \textit{infer} that it must also hold from time ~$t_1$ onward. We will refer to this, informally, as an \textit{inference property}. This is a key observation that will be used to show some sort of Markov property for the renewal times and thus get the i.i.d. structure.

Going back to the classical contact process, in studying the behaviour of the rightmost particle as in \cite{kuczek1989central}, one considers a sequence of times at which the process—if restarted from the rightmost occupied site alone—would still survive. In effect, this amounts to evaluating the classical contact process under its most hostile conditions: all individuals to the left of the rightmost site are removed, and the question becomes whether the process can survive from this most unfavourable configuration. In that case, once the rightmost particle survives, it dictates the behaviour of all future rightmost individuals, in the sense that they must all be its descendants by a crossing paths argument.

Here, the strategy will be similar, with an important modification that we also ask for some regularity regarding the position of the rightmost particle across a class of configurations. First, we define an event that indicates whether a special property has been verified within a given time gap. This special property is determined by \textit{both} the graphical construction and the configuration of the process at the beginning of the time interval under consideration (which itself depends on the initial configuration). 

In the following, $(\chi_t)_{t \geq 0}$ will designate any process, and we will specify when the results depend on specific processes and their parameters. In order to simplify notation, we will often write~$\chi_\cdot$ in place of the full process~$(\chi_t)_{t \geq 0}$.

\subsection{Special property}

\begin{definition} \label{specialPropertySpont}
    Let $x\in \mathbb{Z}$,~$c\in \mathcal{C}^x$ and~$t_1,t_2\in[0,\infty)$ with~$t_1<t_2$. For a process $(\chi_t^c)_{t\geq 0}$, we define the following event:
\begin{align}
\mathcal{P}(\chi_\cdot,c,[t_1,t_2]) 
= &\;\Big\{ \exists x \in \mathbb{Z} : \chi_{t_1}^c(x) = 1 \Big\} 
\;\bigcap \\ &\Bigg\{ 
\forall d \in \mathcal{C}^{r(\chi^c_{t_1})}, \; 
    \exists y \in \mathbb{Z} \text{ such that } \chi_{t_1,t_2}^{d}(y) = 1 
    \label{FirstPartProperty} \\
&\forall s \in [t_1, t_2], \; 
    r(\chi_s^c) = r(\chi_{t_1,s}^d) 
    \label{SecondPartProperty}
\Bigg\}.
\end{align}
Finally, for a given $t\geq 0$, let $\mathcal{P}(\chi_\cdot,c,[t,\infty))=\cap_{s\geq 0 }\mathcal{P}(\chi_\cdot,c,[t,t+s])$.
\end{definition}
We will sometimes forget the $\chi_\cdot$ and $c$ in the property $\mathcal{P}$ when the context is clear.

An illustration of the event that is being captured by Definition \ref{specialPropertySpont} is made in Figure \ref{FIGspecialPropertySpont}. We refer the first part of the event in line \eqref{FirstPartProperty} as the \textbf{survival property}, whereas the other part of that event in line \eqref{SecondPartProperty} is called the \textbf{invariance property}. In words, they are the following: 

\begin{itemize}
     \item \textbf{Survival Property:} Regardless of how we change the configuration around the rightmost particle at moment $t_1$ by any other configuration with the same position of the rightmost particle, the modified process survives until moment $t_2$.
     
     \item \textbf{Invariance Property:} Again, even if we change the configuration by any other configuration with the same rightmost particle, the position of rightmost particle of all those processes would always agree on the time interval~$[t_1,t_2]$ with the position of the rightmost particle for the original process~$(\chi_t^c)$. Therefore, the trajectory of the rightmost site is somehow invariant under those changes in the configuration.     
\end{itemize}

\begin{figure}[h]
        \centering
        \includegraphics[width=0.9\linewidth]{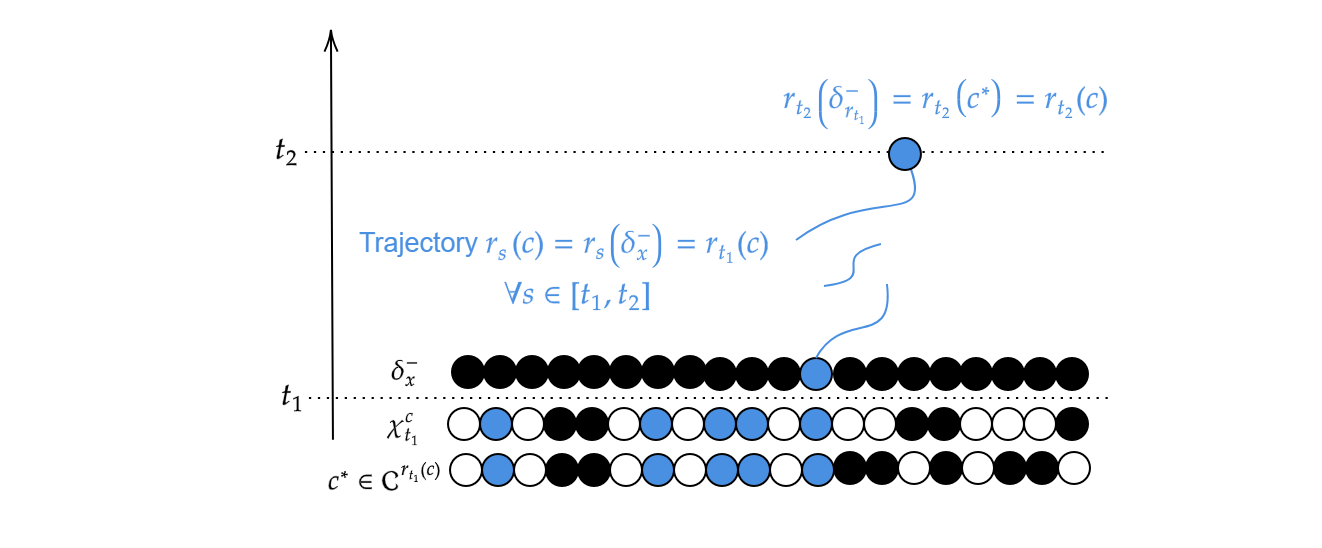}
        \caption{Event $\mathcal{P}(\chi_\cdot,c,[t_1,t_2])$ as in Definition \ref{specialPropertySpont} in case it is not empty}
        \label{FIGspecialPropertySpont}
\end{figure}

Throughout the rest of this work, given~$t\in[0,\infty)$, we will often write~$r_t(c)$ instead of~$r(\chi_t^c)$ to refer to the rightmost~$1$ of the process~$(\chi_t^c)_{t\geq 0}$.  We will also write~$r_{s,t}(d)$ instead of~$r(\chi_{s,t}^d)$ when we want to refer to the rightmost~$1$ of the process~$(\chi_{s,t}^d)_{t\geq s}$. 

\begin{rk}   We make a few remarks regarding the event in Definition \ref{specialPropertySpont}:

    \begin{enumerate}
        \item For $t_1<t_2<t_3$, we have $\mathcal{P}(\chi_\cdot,c,[t_1,t_3])\subseteq \mathcal{P}(\chi_\cdot,c,[t_1,t_2])$ but $\mathcal{P}(\chi_\cdot,c,[t_2,t_3])\nsubseteq \mathcal{P}(\chi_\cdot,c,[t_1,t_3])$.
 
        \item For Spont model (which is attractive), the survival property can be reformulated more simply as $\{\exists y \in \mathbb{Z} \text{ such that } \xi_{t_2}^{\delta_{r_{t_1}(c)}^-}(y) = 1\}$.
        
        \item If the survival property holds for Spont model, then, by coupling, it holds for IS model. 
      
        \item If the invariance property holds then for any $d,d'\in\mathcal{C}^{r_{t_1}(c)}$ we have that
        \[
        \forall s\in [t_1,t_2],~r_{t_1,s}({d}) = r_{t_1,s}({d'})=r_{t_1,s}(\delta_{r_{t_1(c)}}^-)=r_s(c).\] 
    \end{enumerate}
\end{rk}

\begin{lem} \label{InferenceProperty}
    For $\mathcal{P}$ as in Definition \ref{specialPropertySpont}, the following is true:

    \begin{equation*}
        \mathcal{P}(\chi_\cdot,c,[t_1,t_2])\cap \mathcal{P}(\chi_\cdot,c,[t_2,\infty))\subset \mathcal{P}(\chi_\cdot,c,[t_1,\infty))
    \end{equation*}
\end{lem}

\begin{proof} Let~$d\in  \mathcal{C}^{r_{t_1}(c)}$. 
    We first verify how the survival property carries over. Thanks to~$\mathcal{P}([t_1,t_2])$ we know that there exists~$y\in \Z$ such that~$\chi_{t_1,s}^{d}(y) = 1 $ for every~$s\in[t_1,t_2]$. For~$s>t_2$, we have that~$\chi_{t_1,s}^{d}=\chi_{t_2,s}^{d^*}$ for all~$s\geq t_2$ with~$d^*=\chi_{t_1,t_2}^{d}$. Thanks to~$\mathcal{P}([t_1,t_2])$ again, we have that~$r(d^*)=r_{t_2}(c)$ so~$d^*\in\mathcal{C}^{r_{t_2}(c)}$ and applying the survival property of~$\mathcal{P}([t_2,\infty))$ we obtain that there exists~$y\in \Z$ such that $\chi_{t_2,s}^{d^*}(y) =1=\chi_{t_1,s}^{d}(y)$.
    
    We now move on to the invariance property.  We aim to prove that ${r}_{t_1,s}(d)={r}_{s}(c)$ for every $s\geq t_1$. It is again obvious for $s\in[t_1,t_2]$. For $s>t_2$, as we previously wrote, $r_{t_1,s}(d)=r_{t_2,s}(d^*)$ with $d^*\in\mathcal{C}^{r_{t_2}(c)}$, so, applying the invariance property of $\mathcal{P}([t_2,\infty))$ we obtain that $r_{t_1,s}(d)=r_{t_2,s}(d^*)=r_{s}(c)$.   

\end{proof}

As a preview of our approach, we are looking for am moment when the special property we defined is satisfied forever. To find it, we introduce the following stopping time which is a central object in our analysis: it allows us to interrupt the process when the special property fails to hold.

\begin{definition} \label{defFSpont}
    Let $x\in \mathbb{Z}$ and let $c\in \mathcal{C}^x$. We consider the following stopping time:
\begin{align} \label{DefOfF}
        F(\chi_\cdot,c) 
        &= \inf \{s\geq 0\,:\,\mathcal P (\chi_\cdot,c,[0,s])\text{ is not satisfied}\} \\
       &= \inf \{s\geq 0\,:\,\exists d \in \mathcal{C}^x, \forall y\in \mathbb{Z}, \chi_s^{d}(y)\neq 1\, \text{ or } r_s({d}) \neq r_s({c})\} 
    \end{align}

\end{definition}

\begin{lem} \label{cDoesNotMatterFInfty}
    Once $x\in \mathbb{Z}$ is fixed, the occurrence of the event $\{F(\chi_\cdot,c)=\infty\}$ does not depend on the initial configuration $c\in\mathcal C ^x$ (so we can write $\{F(\chi_\cdot,x)=\infty\}$).
\end{lem}
\begin{proof}
    Let $c\in\mathcal{C}^x$ such that $F(\chi_\cdot,c)=\infty$. Let us prove that for $c'\in \mathcal{C}^x$, we also have $F(\chi_\cdot,c')=\infty$. It is clear how the survival property transfers from one process to another because it only depends on $x$; as for the inference property, thanks to $\mathcal{P}(\chi_\cdot,c,[0,\infty))$, we know that $r_s(c')=r_s(c)$ for any s, so the invariance property holds also for $c'$. 
\end{proof}

\begin{rk}
    We have just concluded that~$\mathbb P(F(\chi_\cdot,c) =\infty)$ is the same across all the initial configurations~$c\in\mathcal{C}^x$. Note also that the choice of~$x\in \mathbb{Z}$ also does not matter for the value of~$\mathbb P(F(\chi_\cdot,x) =\infty)$ due to translation invariance of the graphical construction. 
\end{rk}

We are now ready to present our main results, namely that the special property has a positive probability for the processes we are interested in, and if this property fails, it does so quickly.

\begin{prop} \label{FPossibleAndSmallCluster}
    Let $x\in \mathbb{Z}$ and $c\in \mathcal{C}^x$. Consider~$(\xi^c_t)_{t\geq0}$ a Spont process with parameters~$\lambda>\lambda_c$ and~$p>\tilde{p}$ as in Remark~\ref{SupercriticalSpont}. The following is true:
    \begin{enumerate}
        \item $\mathbb{P}(F(\xi_\cdot,c)=\infty)>0$ uniformly in~$c\in\mathcal C^x$,
        \item For any~$t_0\in[0,\infty)$, it follows that $\mathbb{P}(t_0<F(\xi_\cdot,c)<\infty)\leq Ae^{-B~t_0}$ for some choice of constants $A,B>0$ that do not depend on the initial configuration $c\in\mathcal C^x$. 
    \end{enumerate}
    Moreover, the same results hold if we replace~$\mathbb P$ by the measure conditioned on survival ${\mathbb P}(\cdot\,\mid\,{\mathcal S_\xi(c))}$.
\end{prop}

\begin{prop} \label{PropertiesFailureTimeCPIS}
    Let $x\in\mathbb{Z}$ and $c\in \mathcal{C}^x$. Consider~$(\eta_t^c)_{t\geq 0}$ an IS process with parameters~$\lambda>\lambda_c$ and~$p>\tilde{p}$ as in Remark~\ref{SupercriticalSpont}. Then the following is true:
\begin{enumerate}
        \item $\mathbb{P}(F(\eta_\cdot,c)=\infty)>0$  uniformly in $c\in \mathcal C^x$,
        \item for any~$t_0\in[0,\infty)$, it follows that~$\mathbb{P}(t_0<F(\eta_\cdot,c)<\infty)\leq A'e^{-B' t_0^p}$ for some choice of constants $A',B',p>0$ that do not depend on the initial configuration $c\in\mathcal C^x$.
    \end{enumerate}
Moreover, the same results hold if we replace~$\mathbb P$ by the  measure conditioned on survival ${\mathbb P}(\cdot\mid \mathcal{S}_\eta(c))$.
\end{prop}

The results are very similar, but we have a stretched exponential estimate for the second one. Their proofs are  postponed to Section~\ref{Sec_Prop}, since our immediate goal is to show how this object can be used to construct the previously mentioned sequence of renewal times.

\subsection{Definition of renewal times}\label{subDefinitionOfTau}

\paragraph{} Our goal is to define a sequence of renewal times for both processes and we will show in Section~\ref{Sec_Markov} that the increments of the position of the rightmost 1 between these times are i.i.d. 

 Let $x\in \mathbb{Z}$ and $c\in \mathcal{C}^x$. We will construct a sequence of times~$(\sigma_n(\chi_\cdot,c))_{n\in\mathbb N_0}$ associated to this process. We may write~$\sigma_n(\chi_\cdot,c)$ simply as~$\sigma_n$ for when the process and its initial configuration are clear from the context. 

We will define $\sigma_0$ the first element of this sequence with the help of a sequence~$(F_n^0)_{n\in\mathbb N_0}$ of stopping times in the following way. Let~$F_0^0=0$. Define:
\begin{align*}
    F_1^0(\chi_\cdot,c) = \inf \big\{ s \geq 0\,:\, \mathcal{P}(\chi_\cdot,c,[0, s])\text{ is not satisfied}
    \big\}. 
\end{align*}
 We have two possibilities. If $F_1^0=\infty$, we let $\sigma_0 = F_0^0=0$ and complete the sequence of stopping times by making  $F_2^0=F_3^0=\dots=\infty$. Otherwise, if $F_1^0<\infty$, we shall try again to search for a moment where the special property is verified. Then, suppose that we have built for some $k\in \mathbb N$ a sequence of failure times $F_1^0< F_2^0< \dots < F_k^0<\infty$. Then, let $F_{k+1}^0$ be the following stopping time: 
\begin{align*}
     F_{k+1}^0(\chi_\cdot,c) = \inf \big\{ s \geq F_k^0\,:\, \mathcal{P}(\chi_\cdot,c,[F_k^0, s])\text{ is not satisfied}
    \big\}. 
\end{align*}

    If $F_{k+1}^0=\infty$, let $\sigma_0 = F_k^0$ and complete the sequence by making $F_{k+2}^0=\dots=\infty$. Otherwise, if $F_{k+1}^0<\infty$, we proceed as before. 
    
    One can then write $\sigma_0$ using the following decomposition in terms of stopping times: 
    \begin{equation} \label{defOfSigmaSpont}
        \sigma_0(\chi_\cdot,c)= \inf\left\{F_{k}^0(\chi_\cdot,c)\,:\,F_{k+1}^0(\chi_\cdot,c) = \infty\right\}.
    \end{equation}
\begin{rk}
    We have that $\left\{\sigma_0(\chi_\cdot,c)=0\right\}=\{F_1^0(\chi_\cdot,c)=\infty\}=\{F(\chi_\cdot,c)=\infty\}.$
\end{rk}

The following result is a consequence of Propositions~\ref{FPossibleAndSmallCluster} and~\ref{PropertiesFailureTimeCPIS} and guarantees that the random time~$\sigma_0$ as above is, conditioned on the survival, finite almost surely for $\chi_\cdot$ being a Spont or an IS process with parameters~$\lambda>\lambda_c$ and~$p>\tilde{p}$ as in Remark~\ref{SupercriticalSpont}. 
\begin{cor} \label{CorSpontRenTimes}
        There exist constants $a,b,p>0$ such that~${\mathbb{P}}\left(t<\sigma_0(\chi_\cdot,c)\mid {\mathcal S_\chi(c)}\right)\leq ae^{-bt^p}$. In particular, this implies that $\sigma_0(\chi_\cdot,c)$ is almost surely finite. 
\end{cor}
    
\begin{proof} 
To simplify the notation in the proof, we will denote by $\mathbb{P}_{\mathcal S}$ the conditioned measure and by ${\mathbb{E}_{\mathcal S}}$ the associated expectation operator. Let~$t\in[0,\infty)$ be given and note that it is enough to show this result for~$t$ large enough as one can then increase the constant~$a$. Let~$k(t)=\max\{k\in\mathbb N\,:\,F_k^0\leq t\}$ be the index of the last failure time before~$t$. We have the following decomposition:

    \begin{equation*}
\mathbb{P}_{\mathcal S}\left(t<\sigma_0(\chi_\cdot,c)\right) \leq 
\underbrace{
  \mathbb{P}_{\mathcal S}\left(\{t<\sigma_0(\chi_\cdot,c)\} \cap \{k(t) \geq  t^{1/4}\}\right)
}_{\text{(1)}} 
+ 
\underbrace{
  \mathbb{P}_{\mathcal S}\left(\{t<\sigma_0(\chi_\cdot,c)\}\cap \{k(t)<t^{1/4}\}\right)
}_{\text{(2)}}
\end{equation*}

We begin with~$(1)$. Since the event in this probability is contained in the event where~$k(t)\geq t^{1/4}$ , we bound the probability of the latter:

\begin{align}
    \mathbb{P}_{\mathcal S}(k(t)\geq t^{1/4}) & \leq \mathbb{P}_{\mathcal S}\left(F_1,\dots,F_{\lfloor t^{1/4} \rfloor }<\infty\right) \\
      &= {\mathbb{E}_{\mathcal S}}\left[{\mathbb{E}_{\mathcal S}}\left[\mathds{1}_{F_1,\dots,F_{\lfloor t^{1/4} \rfloor }<\infty}\mid \mathcal F_{F_{\lfloor t^{1/4} \rfloor -1}}\right]\right] \\ \label{cor_1_inequality1}
    &={\mathbb{E}_{\mathcal S}}\left[\mathds{1}_{F_1,\dots,F_{\lfloor t^{1/4} \rfloor -1}<\infty}{\mathbb{E}_{\mathcal S,c}}\left[\mathds{1}_{F_{\lfloor t^{1/4} \rfloor }<\infty}\mid\mathcal F_{F_{\lfloor t^{1/4} \rfloor -1}}\right]\right] \\ \label{cor_1_inequality2}
    &\leq p{\mathbb{E}_{\mathcal S}}\left[\mathds{1}_{F_1,\dots,F_{\lfloor t^{1/4} \rfloor -1}<\infty}\right] \leq \cdots \leq q^{\lfloor t^{1/4} \rfloor } .
\end{align}

The first inequality in~\eqref{cor_1_inequality2} follows for some~$q<1$ due to the first conclusion of Propositions~\ref{FPossibleAndSmallCluster} and~\ref{PropertiesFailureTimeCPIS}; the final inequality follows by repeating the same procedure iteratively. 

We now treat the term in~$(2)$. The event inside that probability implies that there exists~$j\in\{1,\dots,k(t)+1\}$ such that~$F_j<\infty$ and~$F_j-F_{j-1}>t^{1/4}$. However, we have 
\begin{align*}
    \mathbb{P}_{\mathcal S}&\left(\{F_j-F_{j-1}>t^{1/4}\}\cap\{ F_j<\infty\}\right) = \\
    & = {\mathbb{E}_{\mathcal S}}\left[{\mathbb{E}_{\mathcal S}}\left[\mathds{1}_{\{F_j-F_{j-1}>t^{1/4}\}\cap \{F_j<\infty\}}\mid\mathcal F_{F_j}\right]\right]\leq Ae^{-B(t^{1/4})^p}
\end{align*}
due to the second conclusion of Propositions~\ref{FPossibleAndSmallCluster} and~\ref{PropertiesFailureTimeCPIS}. Thus, by an union bound it follows that~$(2)$ is bounded by~$t^{1/4+1}Ae^{-B(t^{1/4})^p}$, which concludes the proof. 
\end{proof}

Since~$\sigma_0(\chi_\cdot,c)$ is almost surely finite conditionally on survival, we now define the other terms of the sequence~$(\sigma_n)_{n\in\mathbb N_0}$. Suppose that this sequence has been defined until some time~$\sigma_m$ for some~$m\in\mathbb N_0$. We then let~$F_0^{m+1}=\sigma_m+1$ and we define:
\begin{align} \label{def_Generic_F}
    F_1^{m+1}(\chi_\cdot,c) = \inf \big\{ s \geq F_0^{m+1}\,:\, \mathcal{P}(\chi_\cdot,c,[F_0^{m+1}, s])\text{ is not satisfied}
    \big\} 
\end{align}

If~$F_1^{m+1}=\infty$, we let~$\sigma_{m+1}=F_0^{m+1}=\sigma_k+1$ and we complete the sequence by declaring~$F_2^{m+1}=F_3^{m+1}=\dots=\infty$; otherwise, we try again, namely: assume that we have build the sequence of failure times~$F^{m+1}_1<F_2^{m+1}<\dots<F_k^{m+1}<\infty$. Then, let:
\begin{align*}
     F_{k+1}^{m+1}(\chi_\cdot,c) = \inf \big\{ s \geq F_k^{m+1}\,:\, \mathcal{P}(\chi_\cdot,c,[F_k^{m+1}, s])\text{ is not satisfied}
    \big\} 
\end{align*}

If~$F^{m+1}_{k+1}=\infty$, let~$\sigma_{m+1}=F^{m+1}_k$ and complete the sequence declaring~$F^{m+1}_{k+2}=F^{m+1}_{k+3}=\dots=\infty$; otherwise, we go back to the step before. In that way, we can define~$\sigma_{m+1}$ using this sequence of stopping times:

\begin{equation} \label{defOfGenericSigmaSpont}
        \sigma_{m+1}(\chi_\cdot,c)= \inf\{F_{k}^{m+1}\,:\,F_{k+1}^{m+1} = \infty\}.
    \end{equation}
\begin{prop}\label{Prop_Time_Between_ren}
    There exists constants~$c,C,p>0$ independent of the initial configuration~$c\in\mathcal C^x$ such that for~$\sigma_0,\sigma_1$ as in~\eqref{defOfGenericSigmaSpont} for the Spont process it follows that:
    \begin{equation*}
        \mathbb{P}\left(t<\sigma_1(\xi_\cdot,c)\,\mid\,\sigma_0(\xi_\cdot,c)=0\right) \leq Ce^{-ct^p}
    \end{equation*}
\end{prop}


\begin{prop}\label{Prop_Time_Between_renCPIS}
    There exists constants~$c,C,p>0$ independent of the initial configuration~$c\in\mathcal C^x$ such that for~$\sigma_0,\sigma_1$ as in~\eqref{defOfGenericSigmaSpont} for the inherited sterility process it follows that:
    \begin{equation*}
    \mathbb{P}\left(t<\sigma_1(\eta_\cdot,c)\,\mid\,\sigma_0(\eta_\cdot,c)=0\right) \leq Ce^{-ct^p}
    \end{equation*}
\end{prop}
The proofs of Proposition~\ref{Prop_Time_Between_renCPIS} and~\ref{Prop_Time_Between_ren} are respectively postponed to Section~\ref{SectionMainProofCPIS} and~\ref{SubPropSpont}.

\section{Estimates on the occurrence of the special property} \label{Sec_Prop}

\subsection{Conditions for discrepancies in Spont rightmost positions}

\paragraph{} The key idea behind this section arises from answering the following question: how can the position of the rightmost particle of two Spont processes started from two different configurations in~$\mathcal C^x$ end up in different positions? Thanks to attractiveness, it is enough to compare two processes started from $c\in\mathcal C^x$ and $\delta_x^-$.

 Let $x\in\Z$; for a Spont process started from~$\delta_x^-$ we will denote by
\begin{align*}
    r_t^-=\sup\{y\in\mathbb Z\,:\,\xi_t^{\delta_x^-}(y)=1\}~~\text{ and }~~
    \ell_t^{-}=\inf\{y\in\mathbb Z\,:\,\xi_t^{\delta_x^-}(y)=1\}
\end{align*}
the position of its rightmost and its leftmost site in state~$1$ at moment~$t$, respectively.

\begin{lem} \label{AgreeInsideInterval}
    Let~$x\in \mathbb{Z}$ and~$c\in \mathcal{C}^x$. Assume that~$(\xi_t^{\delta_x^-})_{t\geq 0}$ survives. Then, for any time $t\in [0,\infty)$, we have
    \[  
    \xi_t^{c}(y)=\xi_t^{\delta_x^-},\text{ for all }y\in [\ell_t^-,r_t^-].
    \] 
\end{lem}

\begin{proof}
    We proceed by contradiction so let~$P\in[0,\infty)$ be the first problematic moment where this property does not hold to be true.
    Then, there exists a site~$x_P\in [\ell_P^-,r_P^-]$ such that~$\xi^{\delta_x^-}_P(x_P)\neq \xi^{c}_P(x_P)$, however at any~$s\in [0,P)$ it would be true that~$\xi^{\delta_x^-}_s(y)=\xi^{c}_s(y)$ for all~$y\in [\ell_s^-, r_s^-]$. Note that on the boundary points~$\ell^-_P$ and~$r_P^-$ the processes always agree since by attractiveness~$\xi_t^{\delta_x^-}\lesssim\xi_t^c$ for all~$t\geq 0$. Thus,~$x_P\in(\ell_P^-,r_P^-)$. Note also that this moment~$P$ must correspond to an arrival moment of~$\mathcal H$ affecting the site~$x_P$, otherwise processes would continue to agree on that site. 

   The blocking marks and healing marks, which turn state site respectively to -1 and to 0, affect both processes simultaneously, so they cannot be the source of a disagreement. So there is a transmission arrow affecting $x_P$ at time $P$ and one has~$\xi_P^c(x_P)=1$ but~$\xi_P^{{\delta_x^-}}(x_P)\neq 1$ because the other way around is not possible due to attractiveness. By definition of $P$, both processes agree on $x_P$ before $P$ so $\xi_{P^-}^c(x_P)=0=\xi_{P^-}^{{\delta_x^-}}(x_P)$. And the origin $y_P$ of the transmission arrow is a neighbour of $x_P$, so $y_P\in[\ell_P^-,r_P^-]$ and $\xi_{P^-}^c(y_P)=1=\xi_{P^-}^{{\delta_x^-}}(y_P)$. So it leads to a contradiction because both processes turn to $1$ at $x_P$. 

\end{proof}

\begin{lem}\label{split_site_never_visited}
    Let~$x\in \mathbb{Z}$ and~$c\in \mathcal{C}^x$. Assume that~$(\xi_t^{{\delta_x^-}})_{t\geq 0}$ survives. Let:
    \begin{equation*}
        D = \inf\{t\geq 0\,:\,r(\xi_t^{\delta_x^-})\neq r(\xi_t^c)\}
    \end{equation*}
    be the first moment where the position of their rightmost occupied site differ for the processes started from~$c$ and from~$\delta_x^-$. Then,~$D$ is must be a moment where the rightmost particle of both processes attempts to jump to a site that has never seen a healing mark until time~$D$.
\end{lem}

\begin{proof}
    First note that~$D$ must be a moment of change in the position of the rightmost particle for at least one of the two processes, and let $d-1=r_{D^-}(c)=r_{D^-}^-$ be the last site where their rightmost occupied site agreed on both processes. This change can be due to one in between the following two reasons:

    \begin{enumerate}
        \item Either there exists a healing mark at~$d-1$ at moment~$D$. 
        \item Or there exists a reproduction arrow from $(d-1,D)$ to $(d,D)$.    
        \end{enumerate}

In the first case, due to Lemma~\ref{AgreeInsideInterval}, their rightmost particle at moment~$D$ would be the same as it would belong to a region both processes agreed on~$D^-$, which contradicts the definition of $D$.

In the second case, we have that the rightmost particle successfully managed to jump to the right in one process but not on the other. Note that this is only possible if the destination site was empty in the process started from~$c$ but blocked in the process started from~$\delta_x^{-}$, since the other way around is not possible due to attractiveness.

Let~$\epsilon>0$ (depending on $\mathcal H$) be such that~$r_s^{-}=r_s(c)=d-1$ for all~$s\in [D-\epsilon,D)$. Suppose that there exists~$t^*\in [0,D-\epsilon]$ the last time this site $d$ was occupied in~$(\xi_t^c)_{t\geq 0}$ (in particular, note that~$t^*$ is not arbitrarily close to~$D$). Let us prove that such
$t^*$ cannot exist. Since~$t^*$ was the last time the site~$d$ was occupied by the process started from~$c$, it must be the case that there exists a healing mark in~$(d,t^*)$; and, in between times~$t^*$ and~$D^-$, this site was not occupied in none of the process; indeed, due to attractiveness, if it was occupied in the process started from~$\delta_x^{-}$ it would also be occupied in the process started from~$c$. Therefore, the state of the site~$d$ at time~$D$ depends only on the sequence of healing and blocking marks in the time gap between~$t^*$ and~$D$, and this is the same for both processes as we have constructed them using the same graphical construction. Thus,~$d$ at time~$D$ is on the same state for both processes, which contradicts the definition of~$D$. Therefore,~$t^*$ does not exist and the destination site has never been occupied before in any of those two processes.

Thus,~$D$ is a moment where the rightmost particle of both processes attempts to jump to a site that has never been occupied by either of these processes and that site was blocked in $(\xi_t^{{\delta_x^-}})_{t\geq 0}$ but empty in $(\xi_t^{c})_{t\geq 0}$. Finally, we can conclude that this site has never encountered a healing mark. 
\end{proof}

\subsection{Proof of Propositions \ref{FPossibleAndSmallCluster} and~\ref{Prop_Time_Between_ren}} \label{SubPropSpont}

From the previous lemma, we have that the position of the rightmost particle of the processes~$(\xi_t^c)_{t\geq 0}$ and~$(\xi_t^{\delta_x^-})_{t\geq 0}$ can only differ if they eventually attempt to jump to a site that has never encountered a healing mark before. To prevent this from happening so that the special property can be realised, we will look for some good event under which this cannot occur and therefore the position of their rightmost particle would coincide at all times.

For $x\in \mathbb{Z}$ and $t\in[0,\infty)$, we define  
\begin{equation}\label{def_Htx}
    H_t(x) := \inf\{y\in \mathbb{Z}\,:\,y> x\text{ and }U^y\cap[1,t]=\emptyset\}
\end{equation}
the position at time $t$ of the first site on the right of the site $x$ that has never seen a healing mark in the interval $[1,t]$. In the following lemma, we control the position of this site over time.

\begin{lem} \label{sitesWellBehavedHealingMark}
Let $x\in \mathbb{Z}$ and $t\in[0,\infty)$. Then, for any $\alpha\in[0,\infty)$ one has:

\begin{enumerate}
    \item for any $\epsilon>0$, there exists~$L_1=L_1(\epsilon)\in [0,\infty)$ large enough such that:
    \begin{equation*}
        \mathbb{P} (H_t(x) > x-L_1+4\alpha t \text{ for all }t\geq 1)>1-\epsilon,
    \end{equation*}
    \item  there exists a constant $C$ (independent of~$x$) such that for any~$t_0\in[0,\infty)$:
    \begin{equation*}
        \mathbb P (H_t(x) \leq x+ 4\alpha t \text{ for some }t\geq t_0)\leq Ce^{-t_0/2}.
    \end{equation*}
\end{enumerate}
\end{lem}

\begin{proof}
    Let $\alpha$ and $\epsilon>0$ be given. We start by defining a sequence of auxiliary events in the following way. For any~$n\in \mathbb{N}$, let~$\mathcal{A}_n := \{{H}_n(x) > x+4\alpha(n+1)\}$. First, we observe that~$\mathcal A_n\subset\{{H}_n(x)  > x-L+4\alpha(n+1)\}$ for any~$L\in [0,\infty)$. We also have that~$\mathcal{A}_n\subset \{{H}_s(x)  > x-L+4\alpha s \text{ for }s\in [n,n+1]\}$ since~${H}_s(x) \geq {H}_n(x) $ for all~$s\geq n$. We now bound~$\mathcal{A}_n^c$ in the following way: 

\begin{align*}
{\mathbb{P}}(\mathcal{A}_n^c)
&\leq \mathbb{P}\left({H}_n(x) \leq x+4\alpha (n+1)\right)\\
&\leq \mathbb{P}\left(\bigcup_{y=x}^{x+\lceil 4 \alpha (n+1) \rceil}U^y\cap[1,n]=\emptyset\right) \\
&\leq \mathbb{P}\left(\bigcup_{y=x}^{\lceil 4\alpha (n+1) \rceil} \{X_y \geq n-1\} \text{ where } X_y \overset{\text{i.i.d.}}{\sim} \text{Exp}(1)\right) \\
&\leq \left(\lceil 4\alpha (n+1) \right) \mathbb{P}(X_1 \geq n-1) \\
& \leq C\left(\lceil 4 \alpha (n+1) \rceil\right) e^{-n} \leq C\cdot 8\alpha \cdot e^{-n/2}.
\end{align*}

Let~$n_0\in \mathbb{N}$ be large enough so that
\[
{\mathbb{P}}\left(\bigcup_{n \geq n_0}\mathcal{A}_n^c\right)\leq \sum_{n\geq n_0}C\cdot 8\alpha\cdot e^{-n/2}\leq  \epsilon.\] 

Finally, let $L_1=4\alpha(n_0+1)$. Then, one has
\[H_s(x)\geq x+1\geq x-L_1+ 4\alpha(n_0+1) \geq x-L_1+4\alpha s\text{ for all }s\in [1,n_0+1].
\]
Thus, with this choice of~$L_1$, the event~$\bigcap_{n\geq n_0}\mathcal{A}_n$ implies~$H^x(t)\geq x-L+4\alpha t$ for all~$t\geq 1$, and this event has probability larger than~$1-\epsilon$. This concludes (1).

In order to prove (2), we proceed in a similar manner. Indeed, let~$t_0\in [0,\infty)$ be given. Since~$\bigcap_{n\geq \lfloor t_0\rfloor}\mathcal{A}_n$ implies the event where~$\{H_s(x)>x+4\alpha s \text{ for all }s\geq \lfloor t_0\rfloor\}$, one has that: 
\begin{align*}
    &\mathbb{P} \left( H_t(x) \leq x + 4\alpha t \text{ for some } t \geq t_0\right)
    \leq \mathbb{P} \left( \bigcup_{n \geq \lfloor t_0 \rfloor} \mathcal{A}_n^c \right) \leq \sum_{n \geq \lfloor t_0 \rfloor} C\cdot 8\alpha\cdot e^{-n/2}
    \leq C' e^{-t_0/2},
\end{align*}
which concludes the proof of (2).  \qedhere

\end{proof}

We will need to control how far the rightmost particle of the Spont process could have been; in order to do that, we will use a comparison with the classical contact process.

\begin{lem} \label{rightmostBehavedClassicalCP}
    Consider~$(\zeta_t^{h_x})_{t\geq 0}$ a classical contact process started from the Heaviside configuration~$h_x=\mathds{1}_{(-\infty, x]}$. Let~$\alpha=\alpha(\lambda)$ be the speed of the contact process as in~\eqref{speedCP}. Let~$\bar{R}_t^{h_x}=\sup\{r(\zeta_s^{h_x})\,:\,s\in[0,t]\}$. Then, the following is true:

    \begin{enumerate}
        \item for all~$\epsilon>0$ there exists~$L_2\in [0,\infty)$ (depending only on~$\epsilon$) such that:
        \begin{equation}
            \mathbb{P}( \bar{R}_t^{h_x}<x+L_2+3\alpha t \text{ for all }t\geq 0)>1-\epsilon,
        \end{equation}
        \item there exist constants $c,C>0$ (independent of~$x$) such that for any~$t_0\in [0,\infty)$:
        \begin{equation}
            \mathbb{P}( \bar{R}_t^{h_x} \geq x + 3 \alpha t\text{ for some }t\geq t_0) \leq Ce^{-c{t_0}}.
        \end{equation}
    \end{enumerate}
Moreover, if we assume that $\mathbb P( \mathcal{S}_{\xi}(\delta_{x}^{-}))>0$ then the previous inequalities are true for the conditioned measure $\mathbb P(\cdot \mid \mathcal{S}_{\xi}(\delta_{x}^{-})).$
\end{lem}

\begin{proof}
    The results come from linearity estimates obtained by Durrett and Griffeath~\cite{durrett_griffeath1983supercritical}. We first prove~$(1)$. Let~$\epsilon>0$ be given. Since~$r(\zeta_t^{h_x})/t\rightarrow\alpha$ almost surely as~$t\rightarrow\infty$, we can find~$t_0\in[0,\infty)$ large enough so that~$\mathbb{P}(r(\zeta_t^{h_x})\geq 3\alpha t \text{ for some }t\geq t_0)\leq \frac{\epsilon}{2}$. Let~$L$ be large enough so that the probability of a Poisson random variable with parameter~$\lambda t_0$ being larger than~$L$ is smaller than~$\frac{\epsilon}{2}$. We then have that
    \begin{align*}
    \mathbb{P}(\exists t\geq 0, ~r(\zeta_t^{h_x})\geq x+L+3\alpha t) &\leq 
        \mathbb{P}(\exists t\in[0,t_0],~ r(\zeta_t^{h_x})\geq x+L+3\alpha t)\\ 
        &~~+ \mathbb{P}(\exists t\geq t_0, ~r(\zeta_t^{h_x})\geq x+L+2\alpha t )\\
        &\leq \frac{\epsilon}{2}+\frac{\epsilon}{2},
    \end{align*}
by the choices of~$L$ and $t_0$. Thus,
\[
\mathbb P(\forall t\geq 0,\bar{R}_t^{h_x} <x+L +3\alpha t)=\mathbb{P}(\forall t\geq 0, r(\zeta_t^{h_x})< x+L+3\alpha t)<1-\epsilon.
\]

 For the second point, we directly use the result of Theorem 4 in~\cite{durrett_griffeath1983supercritical}:
 \[
\mathbb P(\exists t\geq t_0,\bar{R}_t^{h_x} \geq x+3\alpha t)=\mathbb{P}(\exists t\geq t_0, r(\zeta_t^{h_x})\geq  x+3\alpha t)\leq Ce^{-ct_0}.
\]
\end{proof}

Finally, we will need a result regarding the extinction time of a Spont process. In~\cite{garetmarchand2014bacteria}, they authors introduced a bacteria model in which the particles can be of type 1 or 2. Both types propagate as oriented percolation with respected parameters $p_1$ and $q_2$ but the 2's block the 1's and the 2's also appear spontaneously with parameter $\alpha$. Using a Bezuidenhout-Grimmett type construction, they proved the survival of the $1's$, a small cluster property (that is if the process of 1's dies out, it is exponentially fast) and the almost linear growth of the $1's$). Considering that our -1 are their 2's, the Spont process is a continuous version of the particular case where $q_2=0$. So we can deduce from their work the following results:
\begin{rk}[Theorem 2 of~\cite{garetmarchand2014bacteria}]\label{smallClusterSPON}
    For a Spont process with parameters~$\lambda>\lambda_c$ and~$p>\tilde{p}$ as in Remark~\ref{SupercriticalSpont}, there exists positive constant $A,B,C$ such that for every $c\in\mathcal C$ 
    \begin{align}
        \mathbb{P}(\tau_\xi(c)=\infty,~t^c(x)>C\|x\|+t)\leq Ae^{-Bt}\\
         \mathbb{P}(t<\tau_\xi(c)<\infty)\leq Ae^{-Bt}
    \end{align}
where $t^c(x)=\inf\{t\geq 0, \xi_t^c(x)=1\}$.
\end{rk}
As we have mentioned before, we are interested in the behaviour of the rightmost occupied site over time. In the previous result, $t^c(x)$ is also $\inf\{t\geq 0, r(\xi_t^c)=x\}$ so we have that the behaviour of the rightmost occupied site is at least linear.

We are finally ready to prove Proposition~\ref{FPossibleAndSmallCluster}. The idea behind the proof is the following: we will search for good events controlling the healing marks ahead of a given location and also the behaviour of the rightmost particle of the Spont process via a comparison with a classical contact process. Because healing marks appear much faster than the classical contact process moves, we are able to find an event where the situation described in Lemma~\ref{split_site_never_visited} cannot happen. This strategy is illustrated in Figure~\ref{FBIggerZero}.

\begin{figure}[h]
    \centering
    \includegraphics[width=1.0\linewidth]{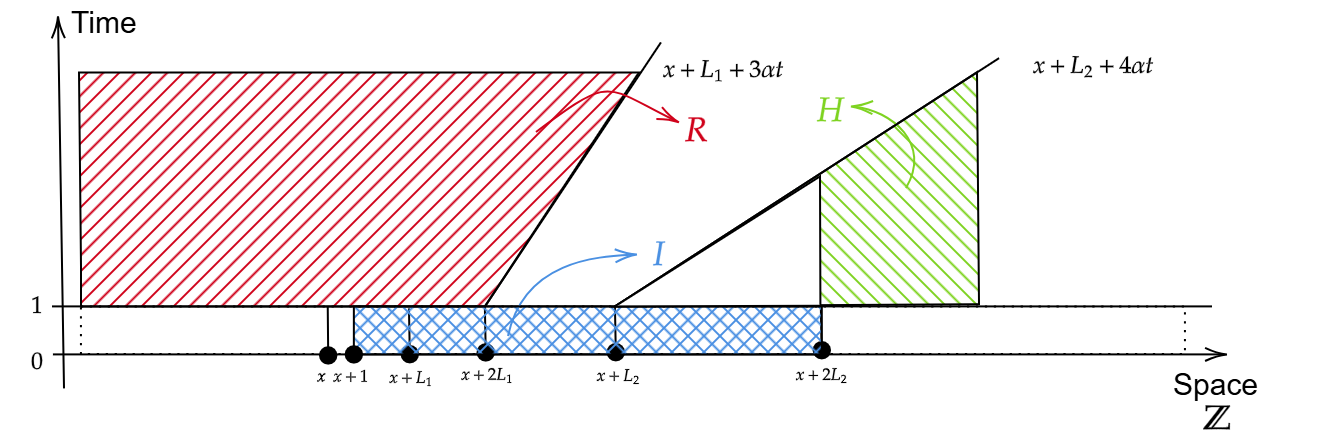}
    \caption{Illustration of argument used in Proposition~\ref{FPossibleAndSmallCluster}. The rightmost particle stays inside the red region due to event~$R$, and all the sites above the green region have seen a healing mark due to event~$H$; the beginning of the dynamics is controlled by the good event~$I$ that appears in blue, where all the sites have seen a healing mark.} 
    \label{FBIggerZero}
\end{figure}

\begin{proof} [Proof of proposition~\ref{FPossibleAndSmallCluster}]
   Consider~$(\xi_t^c)_{t\geq 0}$ a Spont process started from~$c\in\mathcal C^x$ constructed with a graphical construction~$\mathcal H$ as in~\eqref{FormOfGraphicalConstruction} and let~$\alpha=\alpha(\lambda)$ be the speed of the contact process as in~\eqref{speedCP}. 

  Let $(\zeta_{1,t}^{h_{x+1}})_{t\geq 1}$ be a classical contact process constructed with $\mathcal H[1,\infty)$ started from the Heaviside configuration~$h_{x+1}=\mathds{1}_{(-\infty,x+1]}$. For a given~$L_1\in[0,\infty)$, let~$R$ be the event where the rightmost occupied site of this auxiliary contact process is well-behaved in the following sense:
  \[
  R:=\{\forall t\geq 1, \bar{R}_t^{h_{x+1}}<x+L_1+3\alpha t\;\};\]
  using Lemma~\ref{rightmostBehavedClassicalCP} we can increase~$L_1$ in such a way that~${\mathbb{P}}(R\mid \mathcal{S}_{\xi}(\delta_{x}^{-}))>1-\epsilon$.

    Next, thanks to Lemma~\ref{sitesWellBehavedHealingMark}, let~$L_2\in[L_1,\infty)$ be large enough so that 
    \[
    H:=\{\forall t\geq 1, H_t(x+2L_2)>x+L_2+4\alpha t\;\}\]
    satisfies~$\mathbb{P}(H)>1-\epsilon$.
    
    Finally, let~$I$ be the event where nothing happens at $x$ during 1 unit of time an every site between $x+1$ and $x+2L_2$ see a healing mark during this time interval:
    \[
    I=\left\{N_1^{x,x+1}\cap[0,1]=N_1^{x,x-1}\cap [0,1]=U^x\cap [0,1]=\emptyset\right\} \cap \left\{\forall y\in [x+1,x+2L_2],  U^y\cap[0,1]\neq \emptyset \right\}
    \]

    We claim the following:~$I\cap R\cap H\cap \mathcal{S}_{\xi}(\delta_{x}^{-})\subset \{F(\xi,c)=\infty\}$. Indeed, the survival property is verified trivially because~$\xi_t^{\delta^-_x}\lesssim \xi_t^c$ for all~$t\geq 0$ and for any~$c\in\mathcal C^x$. It remains to check the invariance property. Because of~$I$, we have that this condition is verified in the time gap~$[0,1]$. To check that is also verified for all future times, we proceed as follows. First, note that~$\bar{r}_t(c)=\sup\{r(\xi_s^c)\,:\,s\in[1,t]\}$ is such that~$\bar{r}_t(c)\leq \bar{R}_t^{h_{x+1}}$ because~$\xi_t^c\lesssim \zeta_{1,t}^{h_{x+1}}$. Moreover, on~$R\cap H$, we have that, for all~$t\geq 1$, 
    \[
    \bar{R}_t^{h_{x+1}}<x+L_1+3\alpha t\leq x+L_2+4\alpha t < H_t(x+2L_2).
    \]
    Therefore, for all~$t\geq 1$, the site~$\bar{r}_t+1$ possesses a healing mark on the time interval~$[0,t]$. Because of Lemma~\ref{split_site_never_visited}, this implies the invariance property. 

   Let $\mathcal{S}_{\xi}^1(\delta_{x}^{-})=\{ \forall t\geq 1, \exists y\in\Z, \xi_t^{\delta_x^-}(y)=1\}$ the survival event after time 1. We have that:
    \begin{align*}
        \mathbb{P}(I\cap R\cap H\cap \mathcal{S}_{\xi}(\delta_{x}^{-}))
        & =  \mathbb{P}( I\cap R\cap H\cap \mathcal{S}^1_{\xi}(\delta_{x}^{-}))\\
        &=  \mathbb{P}( I)\times \mathbb{P}(R\cap H\cap \mathcal{S}^1_{\xi}(\delta_{x}^{-}))\\
        &= \mathbb{P}( I)\times  \mathbb{P}( H\cap \mathcal{S}^1_{\xi}(\delta_{x}^{-})| R) \times  \mathbb{P}(R)\\
         &= \mathbb{P}( I)\times  \mathbb{P}( H| R) \times \mathbb{P}( \mathcal{S}^1_{\xi}(\delta_{x}^{-})| R) \times  \mathbb{P}(R)\\
         &= \mathbb{P}( I)\times  \mathbb{P}( H) \times \mathbb{P}( R|\mathcal{S}^1_{\xi}(\delta_{x}^{-})) \times  \mathbb{P}(\mathcal{S}^1_{\xi}(\delta_{x}^{-})).\\
    \end{align*}    
The first equality is obtained because $I$ ensures the survival until time 1, the second one uses the fact that $I$ and $R\cap H\cap \mathcal{S}^1_{\xi}(\delta_{x}^{-})$ depend on disjoint parts of the graphical construction (before and after time 1); the fourth uses the fact that, conditioned on $R$, $H$ and $ \mathcal{S}^1_{\xi}(\delta_{x}^{-})$ depend on disjoints parts of the graphical configuration (the red one and the green one in Figure~\ref{FBIggerZero}) and the last one uses the same fact for $H$ and $R$. Finally we obtain the product of four terms with probability strictly positive.    
    This concludes the proof of~$\mathbb{P}(F(\xi_\cdot,c)=\infty)>0$. 
    
    Now let~$t_0\in[0,\infty)$; we aim to show that~$\mathbb P(t_0<F<\infty)\leq Ce^{-ct_0^p}$. It is enough to show that this holds for all~$t_0\geq 1$, and we do so. Let~$H(t_0),R(t_0)$ be the good events defined by    
    \begin{align*}
    H(t_0):=&\{\forall t\geq t_0,~ H_t(x)> x +4\alpha t\,\},\\
    R(t_0):=&\{\forall t\geq t_0,~ \bar{R}_t(\zeta_t^{h_{x+1}})<x+3\alpha t\}.
    \end{align*}
    
    By the same argument we have used before, we have the following inclusion:
   \[
    \mathcal{S}_{\xi}(c)\cap \mathcal{P}(\xi_{\cdot},c,[0,t_0])\cap H(t_0)\cap R(t_0)\subset \{F(\xi_{\cdot},c)=\infty\}.
\]
Therefore, we have

    \begin{align*}
        \mathbb{P}(t_0<F(\xi_{\cdot},c)<\infty) &= \mathbb{P}( \mathcal{P}(\xi_{\cdot},c,[0,t_0]) \cap F(\xi_{\cdot},c)<\infty)
        \\
        &\leq \mathbb{P}\left( \mathcal{P} (\xi_{\cdot},c,[0,t_0]) \cap \mathcal{S}_{\xi}(c)^c\right)+\mathbb{P}\left( \mathcal{P} (\xi_{\cdot},c,[0,t_0]) \cap H(t_0)^c\right)\\
        &~~~~+\mathbb{P}\left( \mathcal{P} (\xi_{\cdot},c,[0,t_0]) \cap R(t_0)^c\right)
        \\
        &\leq \mathbb{P}(t_0<\tau_{\xi}(c)<\infty)+\mathbb{P}\left(H(t_0)^c\right) +\mathbb{P}\left(R(t_0)^c\right).
    \end{align*}

    The first term has the desired exponential decay due to Proposition~\ref{smallClusterSPON}; the second, due to Lemma~\ref{sitesWellBehavedHealingMark} and the third due to Lemma~\ref{rightmostBehavedClassicalCP}.

We then observe how those results transfer for the conditioned probability measure. Note that, since~$\{F(\xi_\cdot,c)=\infty\}\subset\mathcal{S}_\xi(c)$, ${\mathbb P}(F(\xi_\cdot,c)=\infty\mid {\mathcal S_\xi(c)})>0$. 

Moreover,~${\mathbb P}(t<F(\xi_\cdot,c)<\infty\mid {\mathcal S_\xi(c)}) \leq \frac{\mathbb P(t<F(\xi_\cdot,c)<\infty)}{\mathbb P(\mathcal{S}_\xi(c))}$, and therefore the same sub-exponential bound also holds for this conditioned probability measure.

\end{proof}

At last, we prove Proposition~\ref{Prop_Time_Between_ren}.

\begin{proof}[Proof of Proposition~\ref{Prop_Time_Between_ren}]
    We first observe that the proof of Proposition~\ref{Prop_Time_Between_ren} would follow exactly as the proof of Corollary~\ref{CorSpontRenTimes} if we could show the same results of Proposition~\ref{FPossibleAndSmallCluster} for the measure~$\mathbb{P}(\cdot\mid \sigma_0=0)$. Indeed, it would be enough to show that~$\mathbb{P}(F^1_1(\xi_\cdot,c)=\infty\mid\sigma_0=0)>0$ and the sub-exponential bound of the tail of $F_1^1$ defined as in~\eqref{def_Generic_F}. 

    To do so, we define the following events:
    \begin{align*}
        I'&=\left\{ N_1^{x,x+1}\cap[0,2]=N_1^{x,x-1}\cap[0,2]=U^x\cap[0,2]=\emptyset\right\} \\
        &~~~~~~\cap \left\{ U^y\cap[0,2]\neq \emptyset~\forall y\in\{x+1,\dots,x+2L_2\}\right\} \\
        R'&=\left\{ \forall t\geq 2,~r(\zeta_{2,t}^{h_x})_{t\geq 2} \leq x+L_1+3\alpha t  \right\} \\
        H'&=\{ \forall t\geq 2,~H_t(x+2L_2)>x+L_2+4\alpha t \}
 \end{align*}
where~$L_2$ is given as in the proof of Proposition~\ref{FPossibleAndSmallCluster} and $(\zeta_{2,t}^{h_x})_{t\geq 2}$ is a contact process constructed with~$\mathcal H[2,\infty)$. Then, by Lemmas~\ref{InferenceProperty} and~\ref{split_site_never_visited}, we have the following inclusion 
\[
I'\cap R'\cap H'\cap \tilde{\mathcal{S}}_\xi(\delta_x^-)\subset \{F_1^1(\xi_\cdot,c)=\infty\}\cap \{\sigma_0=0\}
\]
where~$\tilde{\mathcal{S}}_\xi(\delta_x^-)$ is the event where the Spont process $(\xi_{1,t}^{\delta_x^-})_{t\geq 0}$ survives. With the same argument than before, $\mathbb P(I'\cap R'\cap H'\cap \tilde{\mathcal{S}}_\xi(\delta_x^-))>0$ so $\mathbb{P}(F^1_1(\xi_\cdot,c)=\infty\mid\sigma_0=0)>0$.

    The sub-exponential bound would hold since~$\mathbb{P}(t<F_1^1(\xi_\cdot,c)<\infty\mid\sigma_0=0) \leq \tilde{c}\mathbb{P}(t<F_1^1(\xi_\cdot,c)<\infty)$, and we can bound this term in a similar way that we have done to conclude item~$(2)$ of Proposition~\ref{FPossibleAndSmallCluster}.

\end{proof}

\subsection{Conditions for discrepancies in IS rightmost positions}

\paragraph{} We want to answer the same question we have asked ourselves before in Section~\ref{SubPropSpont} but for the contact process with inherited sterility: starting from two different initial configurations in~$\mathcal C^x$, what could happen so that their rightmost fertile individual moves to two distinct positions? The difficulty arises from the fact that in the case of inherited sterility, due to the lack of attractiveness, we do not have a “worst-case” configuration to refer to (as was the case with $\delta_x^-$ for the Spont process). We aim to argue that they would only differ when they attempt to visit a site that started blocked in one of the two processes but empty in the other and that has never seen a healing mark before that attempt moment (in particular, it has never been previously occupied by a fertile individual in either of both processes). To prove so, we first need one auxiliary definition.

\begin{definition} \label{gammaOnTheLeft}
    Let~$\gamma_1,\gamma_2:I\rightarrow \mathbb Z$ be two distinct infection paths (that can be active or not) defined on~$I$ such that~$i=\inf I\in I$. We denote~$d(\gamma_1,\gamma_2)=\inf\{t\in I\,:\,\gamma_1(t)\neq \gamma_2(t)\}$ be the first moment where they differ. If~$\gamma_1(i)<\gamma_2(i)$ or if~$\gamma_1(i)=\gamma_2(i)$ and~$\gamma_1(d)<\gamma_2(d)$, we say that~$\gamma_1$ is on the left of~$\gamma_2$.  
\end{definition}

\begin{lem} \label{RegionTheyAgree}
    Let~$x\in \mathbb{Z}$ and let~$c_1,c_2\in \mathcal{C}^x$. Let~$(\eta_t^{c_1})_{t\geq 0}$ and~$(\eta_t^{c_2})_{t\geq 0}$ be two contact processes with inherited sterility and let~$(\xi^{\delta_x^-}_t)_{t\geq 0}$ be a Spont process. Let $t\in [0,\infty)$ such that:
    \begin{enumerate}
        \item there exists $y\in \mathbb{Z}$ such that $\xi^{{\delta_x^-}}_{t}(y)=1$,
        \item $r_s({c_1})=r_s({c_2})$ for all $s\in [0,t]$.
    \end{enumerate}

    Let $\gamma_{\mathrm{left}}^{t}$ be the leftmost active infection path for the process $(\xi_t^{{\delta_x^-}})_{t\geq 0}$ defined on $[0,t]$ in the sense of Definition~\ref{gammaOnTheLeft}; i.e., if there exists~$\gamma:[0,t]\rightarrow\mathbb{Z}$ such that~$\gamma(t)=\ell_t^-$ and~$\gamma\neq \gamma_{\mathrm{left}}^{t}$, then~$\gamma_{\mathrm{left}}^{t}$ is on the left of~$\gamma$. Then, almost surely, the processes $(\eta_t^{c_1})_{t\geq 0}$ and $(\eta_t^{c_2})_{t\geq 0}$ agree on the state of every site in the region $\mathcal{R}\subset\mathbb{R}^2$ delimited by:
\begin{align*}
    \mathcal{R} 
    &= \{(y,s)\in \mathbb{R}^2\,:\, s\in [0,t] \text{ and } y\in [\gamma^t_{\text{left}}(s), r_s(c_1)]\}. 
 \end{align*}
\end{lem}

\begin{proof}
Assumption (1), together with Remark~\ref{CouplingMainThm}--which ensures that~$\xi^{\delta_x^-}_t \lesssim \eta_t^{c_1}, \eta_t^{c_2}$ for all~$t\in[0,\infty)$--implies that the random variables appearing in condition (2) are finite at all times~$s\in[0,t]$. By contradiction, assume that this is not true and let $(x_P,P)$ be the first space-time point where this property fails, with~$P\in [0,t]$. Note that the active infection path $\gamma^t_{\mathrm{left}}$ is also an active infection path for the processes~$\eta_t^{c_1}$ and~$\eta_t^{c_2}$. 
So, at the failure time~$P$, those processes agree on~$\gamma_{\mathrm{left}}^{t}(P)=: L$, and they all have value 1. By assumption~$(2)$, they also agree on~$r_P({c_1})=r_P({c_2})\equiv R$. Therefore, $\eta_P^{c_1}$ and $\eta_P^{c_2}$ agree on the boundary points of the interval~$[L,R]$ and $x_P\in (L,R)$. 

   Then, almost surely there exists~$\epsilon>0$ such that for the interval~$[L,R]$ there exists only one arrival time of~$\mathcal H$ in the region~$[L,R]\times [P-\epsilon,P]$ and this arrival time must have affected~$x_P$; otherwise, the processes would have continued to agree on this site. If it was a healing mark, this would have affected both processes and hence they would still agree on that site. Therefore, the arrival time is a transmission arrow of type either~$+1$ or~$-1$ and $\eta_{P-}^{c_1}(x_P)=\eta_{P-}^{c_2}(x_P)=0$. Without loss of generality, suppose that the switch occurs in $c_1$ so for some~$y\leq x$ there exists~$\gamma_1$ active infection path connecting~$(y,0)\overset{(\mathcal H,c_1)}{\implies}(x_P,P)$.
    
    But note that $\gamma_{\mathrm{left}}^t$ is also an active infection path for the process~$(\eta_t^{c_1})_{t\geq 0}$ as we have observed before, so by a crossing paths argument it follows that there exists~$\gamma^*$ active infection path for the process~$(\eta_t^{c_1})_{t \geq 0}$ such that~$(x,0)\overset{\mathcal H,c_1}{\implies}(x_P,P)$. Moreover,~$\gamma^*$ is on the left of the region delimited by the rightmost particle at that time, i.e.,~$\gamma^*(s)\in [\gamma_{\mathrm{left}}^t(s), r_s({c_1})]$ for all~$s\in [0,P]$. But then up until time~$t-\epsilon$, the path~$\gamma^*$ depends only on state of sites that those two processes agree on, and thus the site~$x_P$ would occupied at time~$P$ in both processes with same state (either by a fertile or a sterile individual, depending on the type of arrow). 

\end{proof}

\begin{lem}\label{split_case_CPIS}
    Let~$x\in \mathbb{Z}$ and~$c_1,c_2\in \mathcal{C}^x$. Assume that the Spont process~$(\xi_t^{\delta_x^-})_{t\geq 0}$ survives. Let~$(\eta_t^{c_1})_{t\geq 0}$ and~$(\eta_t^{c_2})_{t\geq 0}$ be two contact process with inherited sterility. Let~$D\in [0,\infty)$ be the first moment where the rightmost fertile individual of those two contact processes with inherited sterility are in a different position, i.e.:
    \begin{equation*}
        D=\inf\{t\geq 0\,:\,r_t({c_1})\neq r_t({c_2})\}.
    \end{equation*}

    Then, $D$ corresponds to a moment where the rightmost fertile individual of one of those inherited sterility processes jumped to the right and the destination site has never been visited by a fertile individual in any of those two processes before time~$D$ and this jump was successful in one process but unsuccessful in the other process. 
\end{lem}

\begin{proof}
      Note that $D$ must be a moment where the time axis of the site~$r_{D-}({c_1}) = r_{D-}({c_2})=: R$ undergoes a change, so either there exists a healing mark or a transmission mark of type~$+1$. Let $\epsilon>0$ be such that  the only arrival time of~$\mathcal H$ on the grid $[\ell^-_D, R]\times [D-\epsilon, D]$ is the one in~$R$ at moment~$D$. Assume first that this change was due to a healing mark. In~$D-\epsilon$ the processes agree at all sites in~$[\ell^-_D, R]$ because of Lemma~\ref{RegionTheyAgree}. Furthermore, there exists a fertile individual for both of the inherited sterility processes in that interval: namely at least the one in~$\ell^-_D$ which is a 1 for the three processes by coupling construction.  Thus, at the moment when the rightmost individual hits the healing mark their new rightmost individual will then move to the same position. Therefore, the change must have been due to a reproduction mark of type $+1$ from $R$ to $R+1$ at moment $D$. The destination site is empty in one of the processes and blocked in the other.

    By contradiction, assume that the destination site~$R+1$ has been visited by a fertile individual in some (or both) of the two processes before time~$D$. There exists $t^*\in [0, D)$ being the last time the site $R+1$ was occupied by a fertile individual in at least one of the two processes. Note that it would then be occupied in both processes. Indeed, without loss of generality assume that~$\eta_{t^{*-}}^{c_1}(R+1)=1$; then,~$r_{t^{*-}}^{c_1}\geq R+1$, but both processes agree on $[\gamma_{\mathrm{left}}^{D}(t^*), r_{t^*}({c_1})]$ by Lemma~\ref{RegionTheyAgree}, and since~$R+1\in [\gamma_{\mathrm{left}}^{D}(t^*), r_{t^*}({c_1})]$, this site would also be occupied in~$\eta_{t^*}^{c_2}$. Indeed, note that if~$R+1<\gamma_{\mathrm{left}}^{D}(t^*)$, knowing that there exists an active path from $\gamma_{\mathrm{left}}^{D}(t^*)$ to $\gamma_{\mathrm{left}}^{D}(D)\leq R$, there would be a moment after~$t^*$ where the site~$R+1$ would be occupied by $\xi^-$ so also in both processes $\eta^{c_1}$ and $\eta^{c_2}$, and this contradicts the definition of~$t^*$.
    
    By definition of $t^*$, there exists a healing mark from~$U^R$ at time~$t^*$. Let~$\Tilde{\epsilon}>0$ be such that the only arrival time of~$\mathcal H$ on the space-time box $[\gamma_{\mathrm{left}}^{D}(D), R+1]\times [D-\Tilde{\epsilon}, D]$ is the transmission arrow from~$R$ to~$R+1$ at~$D$. Let~$s_2\in [t^*, D-\Tilde{\epsilon}]$ be the last time the site~$R+1$ was blocked by a sterile individual in either of the two processes. If~$s_2\neq D-\Tilde{\epsilon}$, then site is empty in both processes by definition of~$t^*$ and we are done as the transmission arrow~$(R,D)$ to~$(R+1,D)$ would been successful in both processes.  Suppose then that $s_2=D-\Tilde{\epsilon}$. Assumed without loss of generality that $R+1$ at moment $D-\Tilde{\epsilon}$ was occupied by a sterile individual in the process started from $c_1$; we will conclude then that it would also be occupied by a sterile individual in the process started from $c_2$. 
  
      Without loss of generality, $\eta_{D-}^{c_1}(R+1)=0$ and $\eta_{D-}^{c_2}(R+1)=-1$. Let $s_1\in (t^*, D)$ be the time where the life of the sterile individual in $\eta^{c_1}$ started, i.e., $U^{R+1}\cap{[s_1,D]}=\emptyset$. This would imply the existence of $\gamma_2:[0,s_1)\rightarrow \mathbb{Z}$ active infection path for $(\eta_t^{c_1})_{t\geq 0}$ where $\gamma_2(0)=x$ by a crossing path argument; we have that $s_1\in N_2^{\gamma_2(s_1-), R+1}$ and $\gamma_2(s_1-)$ is equal to $R$ or $R+2$. Also note that $\gamma_2(s)\leq r_s({c_1})$ for all $s\in [0,s_1)$ by definition of $r_{s}(c_1)$ and the fact that $\gamma_2$ is active, and thus the whole trajectory of $\gamma_2$ belongs to a region where both processes agree until time $s_1$ by the Proposition \ref{RegionTheyAgree}, and thus $\gamma_2$ is also an active infection path for $(\eta_t^{c_2})_{t\geq 0}$. Hence, the rightmost particle cannot move from $R$ to $R+1$ at time $T$ in $(\eta_t^{c_2})_{t\geq 0}$ either and we end up with a contradiction. 
\end{proof}

\subsection{Proof of Proposition \ref{PropertiesFailureTimeCPIS} and~\ref{Prop_Time_Between_renCPIS}} \label{SectionMainProofCPIS}

The previous lemma implies that the destination site was originally blocked in one of the two processes but not in the other, and has never seen any healing marks until this moment. This is a similar conclusion as the one we had for the Spont process in Lemma~\ref{split_site_never_visited}. In order to prove Proposition~\ref{PropertiesFailureTimeCPIS}, all we need is an estimate of the long but finite survival time of the IS process.

\begin{rk}\label{smallClusterIS}
    For an IS process with parameters~$\lambda>\lambda_c$ and~$p>\tilde{p}$ as in Remark~\ref{SupercriticalSpont}, there exists positive constant $A,p$ such that for every $c\in\mathcal C$ 
    \begin{align}
         \mathbb{P}(t<\tau_\eta(c)<\infty)\leq Ae^{-t^p}.
    \end{align}
\end{rk}

\begin{proof}

The idea is simple: we will restart the underlying Spont process each time it dies out and conclude by using the fact that it is unlikely that it will survive for a long finite time or that it will restart many times without surviving. Let us define the restarting times, starting by \[
T_1=\min\big(\tau(\xi_\cdot^c),\tau(\eta_\cdot^c)\big)\leq \tau(\eta_\cdot^c).
\]
If $T_1<\tau(\eta_\cdot^c)$ then $T_2=\min\Big(T_1+\tau\Big(\xi_{T_1,\cdot}^{\delta_{r(\eta_{T_1}^c)}^-}\Big), \tau(\eta_\cdot^c)\Big). $ If not, we stop the procedure and $T_2=T_1$. 
Generally, let us define
\[
T_{k+1}=
\begin{cases}
\min\left(T_{k}+\tau\left(\xi_{T_k,\cdot}^{\delta_{r(\eta_{T_k}^c)}^-}\right), \tau(\eta_\cdot^c)\right) &\text{ if } T_k<\tau(\eta_\cdot^c) \\
  T_{k} &\text{ if } T_k=\tau(\eta_\cdot^c).
\end{cases}
\]
Then we define $K=\inf\{i,~T_i=\tau(\eta_\cdot^c)\}$. Now, let us denote by $E(t)=\{t< \tau(\eta_\cdot^c)<\infty\}$ the event that we want to estimate. We have the following:

\begin{align*}
\mathbb P(E(t))&=\mathbb P \left(E(t)\cap \{K>\sqrt{t}\}\right) + \mathbb P \left(E(t)\cap \{K\leq \sqrt{t}\}\right)\\
&\leq \mathbb P ( K>\sqrt{t}) + \mathbb P \left(\exists i \in \{1,\ldots,\sqrt{t}\} \text{ s.t. } \sqrt{t}<\tau(\xi_{{T_i},\cdot}^-)<\infty\right)\\
&\leq q^{\sqrt{t}}+A\sqrt{t}\cdot e^{-\sqrt{t}} \leq Ae^{-\frac{\sqrt{t}}{2}},
 \end{align*}
where $q<1$ is the probability of dying out of a Spont process and where the second bound is obtained thanks to Remark~\ref{smallClusterSPON}.
    
\end{proof}

\begin{proof}[Proof of Proposition~\ref{PropertiesFailureTimeCPIS}]
    Because of the conclusion of Lemma~\ref{split_case_CPIS}, the proof of $\mathbb P(F(\eta_\cdot,c)=\infty)>0$ is the same as the one done for the Spont process. Indeed, using the fact that $\xi_t^c\lesssim \eta_t^c\lesssim \zeta_{1,t}^{h_x+1}$ we have that $I\cap R\cap H\cap \mathcal{S}_{\xi}(\delta_{x}^{-})\subset \{F(\eta,c)=\infty\}$, with the same $I$, $H$ and $R$. For the second part, we can also do the same proof but, because of the conclusion of Remark~\ref{smallClusterIS}, we obtain a stretched exponential estimate instead of an exponential one. 
    \end{proof}
\begin{proof}[Proof of Proposition~\ref{Prop_Time_Between_renCPIS}]
    Follows exactly as the proof if Proposition~\ref{Prop_Time_Between_ren}, with the reference to Lemma~\ref{split_site_never_visited} being replaced by its inherited sterility analogue in Lemma~\ref{split_case_CPIS} and the reference to Proposition~\ref{FPossibleAndSmallCluster} being replaced by Proposition~\ref{PropertiesFailureTimeCPIS}. 
\end{proof}

\section{Markov-type property and i.i.d. structure} \label{Sec_Markov}

\paragraph{} We are now in a position to state and prove a Markov-type property for the time $\sigma_1(\chi_\cdot,c)$ defined for the process~$\chi_{\cdot}$ as in \eqref{defOfGenericSigmaSpont}, conditioning on~$\sigma_0(\chi_\cdot,c)=0$ assuming that this event has positive probability. We first define the following conditioned measure. 

\begin{definition} \label{measureConditionedSigma00}
For $x\in\Z$, and $c\in\mathcal C^x$, let~$\bar{\mathbb P}^x_\chi$ denote the following probability measure:
    \begin{equation*}
        \bar{\mathbb{P}}_\chi^x(~\cdot~)=\mathbb{P}(~\cdot~\,|\,\sigma_0(\chi_\cdot,c)=0)
    \end{equation*}
    
\end{definition}
Note that in Definition~\ref{measureConditionedSigma00},  we have highlighted that this probability depends only on site~$x$ and not on the initial configuration~$c\in\mathcal C^x$, as justified before.

\begin{prop} \label{MarkovProp}
    Let~$x,y\in\mathbb{Z}$ and~$c\in \mathcal{C}^x$. Consider~$(\chi_t^c)_{t\geq 0}$ a process constructed with a graphical construction $\mathcal H$. Let~$\sigma=\sigma_1(\chi_\cdot,c)$ be as defined in \eqref{defOfGenericSigmaSpont}.  For any events $A$ and $B$, it follows that:
    \begin{equation}\label{MarkovEq}
    \bar{\mathbb{P}}^x_\chi\left(\mathcal H_{[0,\sigma]}\in A, r_{\sigma}(c) = y, \mathcal H_{[\sigma,\infty)}\in B\right) = \bar{\mathbb{P}}^x_\chi\left(\mathcal H_{[0,\sigma]}\in A, r_{\sigma}(c) = y\right)\bar{\mathbb{P}}^y_\chi\left(\mathcal H\in B\right)
    \end{equation}
\end{prop}

\begin{proof}
   In order to simplify the notation, in the same way we set $\sigma:=\sigma_1$, we will drop the index~$1$  in the stopping times $(F_k^1)_k$ used in the definition of $\sigma_1$. We start by using the definition of $\sigma$ and rewriting the left-hand side of \eqref{MarkovEq} as:
\begin{align*}
    &\bar{\mathbb{P}}^x_\chi\left(\mathcal H_{[0,\sigma]}\in A,\, r_{\sigma}(c) = y,\, H_{[\sigma,\infty)}\in B\right) \\
    &= \sum_{k=1}^{\infty} \bar{\mathbb{P}}^x_\chi\left(\mathcal H_{[0,F_k]}\in A,\, F_{k}<\infty, r_{F_k}(c) = y,\, F_{k+1}=\infty,\, \mathcal H_{[F_k,\infty)}\in B\right) \\
    &= \sum_{k=1}^{\infty} \mathbb{P}\left(\mathcal H_{[0,F_k]}\in A,\, r_{F_k}(c) = y,\, F_{k+1}=\infty,\, \mathcal H_{[F_k,\infty)}\in B,\, F=\infty\right) \times \frac{1}{\mathbb{P}(F=\infty)}
\end{align*}
where we recall that $\{\sigma_0=0\}=\{F=\infty\}$. Note also that in the first equality we have added the intersection with the event $\{F_{k}<\infty\}$ but we have dropped it for the following line as this can be deduced by the event $\{r_{F_k}(c)=y\}$. 

By Lemma~\ref{InferenceProperty}, we have the following key equality of events:
\begin{equation} \label{specialPropertyAndFailureTimes}
    \{F_{k+1}=\infty\}\cap \{F=\infty\}= \mathcal P ([F_k,\infty)) \cap \mathcal P([0, \infty)) =  \mathcal P ([F_k,\infty)) \cap \mathcal P([0, F_k]).
\end{equation}

Then, it follows that the left-hand side of~\eqref{MarkovEq} is equal to: 
\begin{align*}
    &\bar{\mathbb{P}}^x_\chi\left(\mathcal H_{[0,\sigma]}\in A,\, r_{\sigma}(c) = y,\, H_{[\sigma,\infty)}\in B\right) \\
    &=\sum_{k=1}^{\infty} \mathbb{P}\left(\mathcal{H}_{[0,F_k]} \in A,\, r_{F_k}(c) = y,\, \mathcal{P}\left([0,F_k]\right),\, \mathcal{P}\left([F_k,\infty)\right),\, \mathcal H_{[F_k,\infty)} \in B\right) \times \frac{1}{\mathbb{P}(F = \infty)} \\
    &= \sum_{k=1}^{\infty} \frac{\mathbb{P}\left(\mathcal H_{[0,F_k]} \in A,\, r_{F_k}(c) = y,\, \mathcal{P}([0,F_k])\right)}{\mathbb{P}(F = \infty)} \mathbb{P}\left(\mathcal{P}([F_k,\infty)),\, \mathcal H_{[F_k,\infty)}  \in B\,|\,r_{F_k}(c) = y \right)
\end{align*}
where the last equation follows from the strong Markov property since $F_k$ is a stopping time. 
\begin{align*}
  &\bar{\mathbb{P}}^x_\chi\left(\mathcal H_{[0,\sigma]}\in A,\, r_{\sigma}(c) = y,\, H_{[\sigma,\infty)}\in B\right) \\
  &= \sum_{k=1}^{\infty} \frac{\mathbb{P}\left(\mathcal H_{[0,F_k]} \in A,\, r_{F_k}(c) = y,\, \mathcal{P}([0,F_k])\right)}{\mathbb{P}(F = \infty)} \mathbb{P}\left(\mathcal H_{[F_k,\infty)}  \in B\,|\,r_{F_k}(c) = y, \mathcal{P}([F_k,\infty)) \right)\\
  &~~~~\times \mathbb P( \mathcal{P}([F_k,\infty))\,|\,r_{F_k}(c) = y)\\
     &=\sum_{k=1}^{\infty} \frac{\mathbb{P}\left(\mathcal H_{[0,F_k]} \in A,\, r_{F_k}(c) = y,\, \mathcal{P}([0,F_k]),\, \mathcal{P}([F_k,\infty))\right)}{\mathbb{P}(F = \infty)}\\
     &~~~~\times \mathbb{P}\left(\mathcal H_{[F_k,\infty)}  \in B \mid r_{F_k} = y, \mathcal{P}([F_k,\infty)) = \infty\right)
\end{align*}

We note that~$\mathbb{P}\left(\mathcal H_{[F_k,\infty)}  \in B \mid r_{F_k} = y, \mathcal{P}([F_k,\infty)\right)$ is simply~$\bar{\mathbb{P}}^{y}_\chi\left(\mathcal H \in B\right)$ by the definition of the measure~$\bar{\mathbb P}^y$. Therefore, we can reduce this summation to:
\begin{align*}
&\bar{\mathbb{P}}^x_\chi\left(\mathcal H_{[0,\sigma]}\in A,\, r_{\sigma}(c) = y,\, H_{[\sigma,\infty)}\in B\right) \\
    &\sum_{k=1}^{\infty} \frac{\mathbb{P}\left(H_{[0,F_k]} \in A,\, r_{F_k}(c) = y,\, F_{k+1} = \infty,\, F = \infty\right)}{\mathbb{P}(F = \infty)} \bar{\mathbb{P}}^{y}\left(\mathcal H \in B\right) \\     &=\bar{\mathbb{P}}^{y}_\chi\left(\mathcal H \in B\right) \sum_{k=1}^{\infty} \bar{\mathbb{P}}^{x}_\chi\left(\mathcal H_{[0,F_k]} \in A,\, r_{F_k}(c) = y,\, F_{k+1} = \infty\right) \\
    &=\bar{\mathbb{P}}^{x}_\chi\left(\mathcal H_{[0,\sigma]} \in A,\, r_{\sigma} = y\right) \bar{\mathbb{P}}^{y}_\chi(\mathcal H \in B)
\end{align*}
where in first equality we have used again~\eqref{specialPropertyAndFailureTimes}, in the second equality we have used the definition of the conditional measure $\bar{\mathbb P}^x$. 
\end{proof}

\begin{prop} \label{iidSequence}
    Let $x\in\mathbb Z$ and $c\in\mathcal C^x$. For a process~$(\chi_t^c)_{t\geq 0}$, consider the sequence of times $(\sigma_n(\chi,c))_{n\in \mathbb N_0}$ as defined in~\eqref{defOfGenericSigmaSpont}. Under $\bar{\mathbb{P}}^x_\chi$, the sequences of time increments $(\sigma_{i+1}-\sigma_i)_{i\in\mathbb{N}_0}$ and the sequence of space increments $\left((r_t-r_{\sigma_i})_{\sigma_i\leq t <\sigma_{i+1}}\right)_{i\in\mathbb{N}_0}$ are i.i.d.
\end{prop}

\begin{proof}
    For the increments regarding the time, we note the following. Let $E_i$ for $i=0,1,\dots,k$ be measurable subsets of $[0,\infty)$. It then follows:
\begin{align*}
    &\bar{\mathbb{P}}^x_\chi\left(\sigma_{i+1} - \sigma_i \in E_i \text{ for } i \in \{0,1,\dots,k\} \right) \\
    &= \sum_{z \in \mathbb{Z}} \bar{\mathbb{P}}^x_\chi\left(\sigma_1 \in E_0;\, \sigma_{i+1} - \sigma_i \in E_i \text{ for } i \in \{0,1,\dots,k\};\, \sigma_1 < \infty;\, r_{\sigma_1} = z \right) \\
    &= \sum_{z \in \mathbb{Z}} \bar{\mathbb{P}}^x_\chi\left(\sigma_1 \in E_0;\, \sigma_1 < \infty;\, r_{\sigma_1} = z \right) \bar{\mathbb{P}}^z_\chi\left(\sigma_{i+1}(\delta_z^-) - \sigma_i(\delta_z^-) \in E_i \text{ for } i \in \{0,\dots,k-1\} \right) \\
    &= \bar{\mathbb{P}}^x_\chi\left(\sigma_{i+1} - \sigma_i \in E_i \text{ for } i \in \{1,\dots,k\} \right) \sum_{z \in \mathbb{Z}} \bar{\mathbb{P}}^x_\chi\left(\sigma_1 \in E_0;\, \sigma_1 < \infty;\, r_{\sigma_1} = z \right) \\
    &= \bar{\mathbb{P}}^x_\chi\left(\sigma_1 \in E_0 \right) \bar{\mathbb{P}}^x_\chi\left(\sigma_{i+1} - \sigma_i \in E_i \text{ for } i \in \{1,\dots,k\} \right)
\end{align*}
where in the first equality we have decomposed on the position of the rightmost~$1$ at the moment of the first renewal using that this moment is finite; in the second equality we have applied the Markov property; in the third equality we have used translation invariance and noted that~$\sigma_k(c) = \sigma_{k-1}(\delta_{r_{\sigma_0}}^-)$ for all~$k\geq 1$. 

    Let $E_0,E_1,\dots,E_k$ be measurable sets of $\cup_{t\geq 0}D_{[0,t]}$ the space of finite-time trajectories that are right-continuous with left limits using the same notation as \cite{billingsley2013convergence}. Since under~$\Bar{\mathbb{P}}^x$ we are conditioning on~$\sigma_0=0$, one has the following:
\begin{align*}
    &\bar{\mathbb{P}}^x_\chi\left((r_s - r_{\sigma_i})_{\sigma_i \leq s < \sigma_{i+1}} \in E_i \text{ for } i\in\{0,\dots,k\} \right) \\
    &= \bar{\mathbb{P}}^x_\chi\left((r_s - r_{\sigma_i})_{\sigma_i \leq s < \sigma_{i+1}} \in E_i \text{ for } i\in\{0,\dots,k\}, \sigma_1 < \infty \right) \\
    &= \sum_{z \in \mathbb{Z}} \bar{\mathbb{P}}^x_\chi\left((r_s - r_{\sigma_i})_{\sigma_i \leq s < \sigma_{i+1}} \in E_i \text{ for } i\in\{0,\dots,k\}, \sigma_1 < \infty, r_{\sigma_1} = z \right) \\
    &= \sum_{z \in \mathbb{Z}} \bar{\mathbb{P}}^x_\chi\left((r_s)_{0 \leq s < \sigma_1} \in E_0; (r_s - r_{\sigma_i})_{\sigma_i \leq s < \sigma_{i+1}} \in E_i \text{ for } i\in\{1,\dots,k\}, \sigma_1 < \infty, r_{\sigma_1} = z \right) \\
    &= \sum_{z \in \mathbb{Z}} \bar{\mathbb{P}}^x_\chi\left((r_s)_{0 \leq s < \sigma_1} \in E_0,\sigma_1 < \infty, r_{\sigma_1} = z \right) \bar{\mathbb{P}}^z_\chi\left((r_s(\delta_z^-) - r_{\sigma_i}(\delta_z^-))_{ \sigma_1\leq s < \sigma_{i+1}-\sigma_1} \in E_i \text{ for } i\in\{1,\dots,k\}\right) \\
    &= \bar{\mathbb{P}}^x_\chi\left((r_s - r_{\sigma_i})_{\sigma_i \leq s < \sigma_{i+1}} \in E_i \text{ for } i\in\{1,\dots,k\} \right) \sum_{z \in \mathbb{Z}} \bar{\mathbb{P}}^x_\chi\left((r_s)_{0 \leq s < \sigma_1} \in E_0; \sigma_1 < \infty; r_{\sigma_1} = z \right) \\
    &= \bar{\mathbb{P}}^x_\chi\left((r_s)_{0 \leq s < \sigma_1} \in E_0 \right) \bar{\mathbb{P}}^x_\chi\left((r_s - r_{\sigma_i})_{\sigma_i \leq s < \sigma_{i+1}} \in E_i \text{ for } i\in\{1,\dots,k\} \right)
\end{align*}
where in the first equality we have used the fact that $\sigma_1$ is finite under $\bar{\mathbb{P}}^x$ (Corollary~\ref{CorSpontRenTimes}); in the second equality, we have decomposed the value of $r_{\sigma_1}$ using the fact that we are on the event that $\sigma_1$ is finite. Then, we have isolated the first part of the trajectory and finally we have used the Markov-type property proved in Proposition \ref{MarkovProp}. The next equality follows by translation invariance. Finally, note that we can apply those same steps again to the second term in the last equality in order to obtain that the desired i.i.d. space decomposition.

\end{proof}

We make a small observation regarding the definition of the special property we are chasing when defining $\sigma$. For the inference property to hold, it would be enough to require that the rightmost positions of the original and hostile processes agree at all future times, rather than demanding agreement with such a large class of configurations. But note that this would not lead to an i.i.d. structure in time because of the following: for times where the landscape ahead of the rightmost particle of the original process had many blocked sites, it would be potentially easier for the rightmost particles of the original process and the process started for the hostile environment to coincide compared to the case where the landscape ahead of the original process was more empty. 

\section{Proofs of the main theorems} \label{sec_ProofMainTheorems}

\paragraph{} Let $x\in\Z$ and $c\in\mathcal C ^x$. Throughout the rest of this section, we consider~$(\chi_t^c)_{t\geq 0}$ a process (which can be Spont or IS) started from~$c$ with parameters~$\lambda>\lambda_c$ and~$p>\tilde{p}$ as in Remark~\ref{SupercriticalSpont} and we consider the sequence of renewal times~$(\sigma_n(\chi,x))_{n\in\mathbb N_0}$ defined in Section~\ref{subDefinitionOfTau}.

For a given~$t \in [0, \infty)$, let~$N(t) := \sup\{n \in \mathbb{N} : \sigma_n(\chi, c) < t\}$ denote the index of the last renewal strictly before time~$t$. In what follows, we will need to control both the amount of time~$t-\sigma_{N(t)}$ and also the displacement of the rightmost particle from that renewal time~$\sigma_{N(t)}$ up to time~$t$. This is precisely the role of Lemma~\ref{lem_timeSpaceRenewalBeforet}, whose proof closely follows that of Lemma 2.5 in~\cite{Valesin2010Multitype}. In what follows, we will replace~$\sigma_n(\chi,c)$ and $r(\chi_t^c)$ by simply~$\sigma_n$ and $r_t$ to simplify notation. We recall that $\bar{\mathbb{P}}_\chi^x(\cdot)=\mathbb{P}(\cdot\,|\,\sigma_0(\chi_\cdot,c)=0)$.

\begin{lem}\label{lem_timeSpaceRenewalBeforet}
    Let~$t\in[0,\infty)$ be given. Then, there exist constants~$d,D,p>0$ (independent of the initial configuration~$c\in\mathcal C$) such that the following holds for any~$L\geq 0$:
    \[
    \Bar{\mathbb{P}}^x_{\chi}\left((\sigma_{N(t)+1}-\sigma_{N(t)})\vee (\sup\{|r_s-r_{\sigma_{N(t)}}|\,:\,s\in[\sigma_{N(t)},\sigma_{N(t)+1}]\})>L\right) \leq De^{-dL^p}.
    \]
\end{lem}

\begin{proof}
    For~$t\in[0,\infty)$, consider~$\psi(t)$ the random variable defined by:\[
    \psi(t)= M^{(r_{\sigma_{N(t)}},\sigma_{N(t)})}_{\sigma_{N(t)+1}}\vee (\sigma_{N(t)+1}-\sigma_{N(t)})
    \]
    where~$M^{(x,s)}_t:=\sup\{|y-x|\,:\,(x,s)\overset{\mathcal H}{\Longrightarrow}(y,t)\}$ for~$x\in\mathbb Z$ and~$s,t\in[0,\infty)$ with~$s\leq t$. We aim to show that~$\Bar{\mathbb{P}}^x_\chi(\psi(t)>L)\leq De^{-dL^p}$. We have:
\begin{align}
    \Bar{\mathbb{P}}^x_{\chi} \left(\psi(t)>L\right) 
    &= \sum_{k=0}^{\infty}\Bar{\mathbb{P}}^x_\chi\left(\sigma_k < t,\, \sigma_{k+1} \geq t,\, \psi(t) > L,\right) \notag \\
    &= \sum_{k=0}^\infty \int_0^t \Bar{\mathbb{P}}^x_\chi\left( \left\{\sigma_{k+1} - \sigma_k \geq t-s\right\} \cap \left\{ M_{\sigma_{k+1}}^{(r_{\sigma_k,\sigma_k})} \vee (\sigma_{k+1} - \sigma_k) > L \right\} \right) \Bar{\mathbb{P}}^x_\chi\left(\sigma_k \in ds\right) \notag \\
    &= \sum_{k=0}^\infty \int_0^t \Bar{\mathbb{P}}^x_\chi\left( \left\{\sigma_1 \geq t-s\right\} \cap \left\{ M_{\sigma_1}^{(x,0)} \vee \sigma_1 > L \right\} \right) \Bar{\mathbb{P}}^x_\chi\left(\sigma_k \in ds\right)
    \label{eq:psi_tail_bound}
\end{align}
where we have usedthat the $(\sigma_{i+1}-\sigma_i)$ are i.i.d under $\bar{\mathbb P}_\chi^x$ (by Proposition~\ref{iidSequence}). To deal with the probability of the intersection of events in~\eqref{eq:psi_tail_bound}, we observe the following:
\[\Bar{\mathbb{P}}^x_\chi\left(M_{\sigma_1}^{(x,0)} \vee \sigma_1 > L\right) \leq \Bar{\mathbb{P}}_\chi^x\left(\left\{M_{\sigma_1}^{(x,0)} \vee \sigma_1 > L\right\}\cap \left\{\sigma_1<\frac{L}{2\lambda}\right\}\right) + \Bar{\mathbb{P}}^x_\chi\left(\sigma_1 > \frac{L}{2\lambda}\right).
\]

The first term of the above equation is bounded by the probability of a Poisson random variable with parameter~$\lambda\frac{L}{2\lambda}$ to be larger than~$L$, and therefore is smaller than~$Ce^{-cL}$. The second term is bounded by~$Ce^{-cL^p}$ either due to Proposition~\ref{Prop_Time_Between_ren} or due to Proposition~\ref{Prop_Time_Between_renCPIS}, depending on which process we consider.  Therefore, we have that~\eqref{eq:psi_tail_bound} is bounded above by:
\begin{align}
    &\sum_{k=0}^\infty \sum_{i=1}^{\lceil t \rceil} \int_{i-1}^i \left[ \Bar{\mathbb{P}}^x_\chi\left( \sigma_1 \geq t-s \right) \wedge \Bar{\mathbb{P}}^x_\chi\left( M_{\sigma_1}^{(x,0)} \vee \sigma_1 > L \right) \right] \Bar{\mathbb{P}}^x_\chi\left( \sigma_k \in ds \right) \notag \\
    &\leq \sum_{i=1}^{\lceil t \rceil} \left( Ce^{-c(t-i)^p} \wedge Ce^{-L^p} \right) \sum_{k=0}^\infty \Bar{\mathbb{P}}^x_\chi\left( \sigma_k \in [i-1,i] \right)
    \label{eq:tail_sum_bound}
\end{align}
and note that~$\sum_{k=0}^\infty \Bar{\mathbb{P}}^x_\chi\left(\sigma_k\in[i-1,i]\right) =\Bar{\mathbb{E}}^x_\chi\left[|\{n\in\mathbb N_0\,:\,\sigma_n\in[i-1,i]\}|\right]\leq 1$ since~$\sigma_{n+1}-\sigma_n\geq 1$ for all~$n\in\mathbb N_0$ by definition. Therefore, we can upper bound~\eqref{eq:tail_sum_bound} by~
\[
C\sum_{i=1}^\infty e^{-ci^p}\wedge e^{-cL^p}\leq C\lceil L \rceil e^{-cL^p} + C\sum_{i=\lceil L \rceil +1}^\infty e^{-ci^p}\leq De^{-dL^p}
\]
which concludes the proof. 
\end{proof}

We are now finally ready to prove the main results. We start with the proof of Theorem~\ref{thm_Spont_Speed}.  

\begin{proof}[Proof of Theorem~\ref{thm_Spont_Speed} and~\ref{thm_Speed_IS}]
    Firstly, we note that~$\sigma_0$ and~$r_{\sigma_0}$ are finite almost surely conditioned on survival, and that after the first renewal time, the probability measure conditioned on survival can be replaced by~$\bar{\mathbb{P}}_\chi^x$. In words, we are just ignoring the process until the first renewal time.
    
    We start by defining~$\mu_T:=\bar{\mathbb E}_\chi^x[\sigma_1-\sigma_0]$ and~$\mu_S:=\bar{\mathbb E}_\chi^x[r_{\sigma_1-r_{\sigma_0}}]$ the expected values of the temporal and spatial displacements in between renewals, respectively. In that way, we will obtain that~$\mu_{\mathrm{SP}}$ as in~\eqref{eq_speed_Spont} will be given by the ratio~$\mu_S/\mu_T$.

    To check that, we first observe that~$r_{\sigma_n}/\sigma_n\rightarrow \mu_{\mathrm{SP}}$ as~$n\rightarrow\infty$ almost-surely. Indeed,~$\sigma_0$ is finite, so is~$r_{\sigma_0}$, and thus, rewriting $\sigma_n$ and $r_{\sigma_n}$ as the sum of their i.i.d. increments (Proposition~\ref{iidSequence}), we can use the law of large numbers and we have that~$\sigma_n/n\rightarrow\mu_T$ and~$r_{\sigma_n}/n\rightarrow\mu_S$ as~$n\rightarrow\infty$. Since~$\sigma_n\geq n$ by definition of~$\sigma_n$, we have that~$\sigma_n/n\geq 1$ is bounded away from zero, thus~$(r_{\sigma_n}/n)/(\sigma_n/n)\rightarrow \mu_{\mathrm{SP}}$ as~$n\rightarrow\infty$ a.s.

    We will now work to transfer this convergence along integer times, i.e., we will show that~$r_n/n\rightarrow\infty$ as~$n\rightarrow\infty$ a.s. For~$t\in[0,\infty)$, let~$\sigma(t):=\sigma_{N(t)}=\sup\{\sigma_m\,:\,\sigma_m\leq t\}$ be the time of the last renewal before~$t$. Let~$\epsilon>0$ be given. We will show that:
    \begin{equation} \label{target_SLLN_Spont}
        \bar{\mathbb P}_\chi^x\left(\left|\frac{r_n}{n}-\frac{r_{\sigma(n)}}{\sigma(n)}\right|>\epsilon\right) \leq Ce^{-cn^p}.
    \end{equation}

    Note that~$r_{\sigma(n)}/\sigma(n)\rightarrow\mu_{\mathrm{SP}}$ as~$n\rightarrow \infty$ a.s. as it is simply a subsequence of a convergent sequence. Then, ~\eqref{target_SLLN_Spont} along with the Borel-Cantelli lemma would give us the desired claim. 
    
    Firstly, we have that:
    \begin{equation} \label{Target2_SLLN_Spont}
       \bar{\mathbb P}_\chi^x\left(\left|\frac{r_n}{n}-\frac{r_{\sigma(n)}}{\sigma(n)}\right|>\epsilon\right)\leq  \bar{\mathbb P}_\chi^x\left(\left|\frac{r_n}{n}-\frac{r_{n}}{\sigma(n)}\right|>\frac{\epsilon}{2}\right) + \bar{\mathbb P}_\chi^x\left(\left|\frac{r_n}{\sigma(n)}-\frac{r_{\sigma(n)}}{\sigma(n)}\right|>\frac{\epsilon}{2}\right). 
    \end{equation}

    We first bound the second term in~\eqref{Target2_SLLN_Spont}:
\begin{align*}
    \bar{\mathbb P}_\chi^x\left(\left|\frac{r_n}{\sigma(n)}-\frac{r_{\sigma(n)}}{\sigma(n)}\right|>\frac{\epsilon}{2}\right) 
    &\leq  \bar{\mathbb P}_\chi^x\left(|r_n-r_{\sigma(n)}|>\frac{\epsilon}{2}\sigma(n)\right) \\
    & \leq \bar{\mathbb P}_\chi^x\left(\sigma(n)<\frac{n}{2}\right)+\bar{\mathbb P}_\chi^x\left(|r_n-r_{\sigma(n)}|>\frac{\epsilon}{4}n\right);
\end{align*}

and these terms have the desired sub-exponential bounds due to Lemma~\ref{lem_timeSpaceRenewalBeforet}. Then, we bound the first term in~\eqref{Target2_SLLN_Spont}:
\begin{align*}
    \bar{\mathbb P}_\chi^x\left(\left|\frac{r_n}{n}-\frac{r_{n}}{\sigma(n)}\right|>\frac{\epsilon}{2}\right) 
    &\leq  \bar{\mathbb P}_\chi^x\left(|r_n|(n-\sigma(n)) > \frac{\epsilon}{2}n\sigma(n)\right) \\
    &\leq \bar{\mathbb P}_\chi^x\left(\sigma(n)<\frac{n}{2}\right)+\bar{\mathbb P}_\chi^x\left(|r_n|(n-\sigma(n)) > \frac{\epsilon}{4}n^2\right) \\
    & \leq  \bar{\mathbb P}_\chi^x\left(\sigma(n)<\frac{n}{2}\right) +
    \bar{\mathbb P}_\chi^x\left(\{|r_n|(n-\sigma(n))| > \frac{\epsilon}{4}n^2\}\cap \{|r_n|\leq 2\lambda n\}\right) + \bar{\mathbb P}_\chi^x\left(|r_n|>2\lambda n\right)\\
    & \leq ... + \bar{\mathbb P}_\chi^x\left(n-\sigma(n)>\frac{\epsilon}{8\lambda}n\right).
\end{align*}
The first and second term have the desired bound once more due to Lemma~\ref{lem_timeSpaceRenewalBeforet}. For the third term, just note that it is bounded by the probability of a Poisson random variable with parameter~$\lambda n$ to be larger than~$2\lambda n$. 

    Finally, to conclude that~$r_t/t\rightarrow \mu_{\mathrm{SP}}$ as~
    $t\rightarrow \infty$ we proceed as in the proof of Theorem 2.19 of~\cite{liggett1985interacting} noting that:
    \begin{align*}
        \bar{\mathbb P}_\chi^x\left(\max_{n\leq t\leq n+1}(r_t-r_n)\geq \epsilon n\right) \leq Ce^{-cn^p} \\
        \bar{\mathbb P}_\chi^x\left(\max_{n-1\leq t\leq n}(r_n-r_t)\geq \epsilon n\right) \leq Ce^{-cn^p},
    \end{align*}
    since both terms are bounded by the probability of a Poisson random variable with parameter~$\lambda$ to be larger than~$\epsilon n$. Because once more of the Borel-Cantelli lemma, we have that almost surely:\[
    \frac{r_t}{t}\xrightarrow[t\rightarrow\infty]{\mathrm{a.s.}}\mu_{\mathrm SP}.
    \] 
\end{proof}

We note that convergence in~$L^1$ also holds for the subset of initial configurations that allows us to guarantee the finiteness of the first moment of $r_{\sigma_0}$ (conditioned on survival). 
\begin{prop}\label{prop_convL1Spont}
    Let~$c\in\mathcal C$ be a configuration with finitely many sites in state~$1$. Then for~$\mu_{\mathrm{SP}}$ as in~\eqref{eq_speed_Spont} and for~$\mu_{\mathrm{IS}}$ as in~\eqref{eq_speed_IS}, it follows that:
    \[
    \frac{r(\xi_t^c)}{t}\xrightarrow[t\rightarrow\infty]{}\mu_{\mathrm{SP}} \quad \text{in }L^1 ~~~\text{ and } ~~~\frac{r(\eta_t^c)}{t}\xrightarrow[t\rightarrow\infty]{}\mu_{\mathrm{IS}} \quad \text{in }L^1,
    \]
    conditioned on the respective survival event.
\end{prop}

\begin{proof}
In what follows, we denote the law of the process conditioned on survival as~$\mathbb P_{\mathcal S}$ and $\mathbb E_{\mathcal S}$ be its associated expectation operator.      We start by showing that the family~$\{r_n/n\,:\,n\in\mathbb N\}$ is uniformly integrable. We must show that for any~$\epsilon>0$ there exists~$L\in[0,\infty)$ such that
    \[
    \sup_{n\in\mathbb N}{\mathbb E}_{\mathcal S}\left[\frac{|r_n|}{n}\mathds{1}_{\{r_n/n>L\}}\right]<\epsilon.
    \]
    
    Fix~$n\in\mathbb N$. We observe that, by taking into account only transmission arrows (regardless of their type), the random variable~$|r_n|$ is bounded by~$L_c+J_n$ where~$J_n$ is a Poisson random variable with parameter~$\lambda n$ and~$L_c := \max\{|x|\,:x\in\mathbb Z\,\text{ and }\,c(x)=1\}$. Therefore, we have that:
    \begin{align*}
        {\mathbb E}_{\mathcal S}\left[\frac{|r_n|}{n}\mathds{1}_{\{|r_n|/n>L\}}\right] &\leq {\mathbb E}_{\mathcal S}\left[\left(\frac{|r_n|}{n}\right)^2\right]^{1/2}{\mathbb E}_{\mathcal S}\left[\mathds{1}_{\{|r_n|/n>L\}}\right]^{1/2} \\
        &\leq \left(\frac{{\mathbb E}_{\mathcal S}[(L_c+J_n)^2]}{n^2}\right)^{1/2}{\mathbb E}_{\mathcal S}\left[\mathds{1}_{\{(L_c+J_n)>Ln\}}\right]^{1/2} \\
        &\leq\left( \frac{x^2}{n^2} + \frac{L_c{\mathbb E}_{\mathcal S}[J_n]}{n^2} + \frac{\mathbb{E}_{\mathcal S}[J_n^2]}{n^2}\right)^{1/2} \mathbb P_{\mathcal S}\left(L_c+J_n>Ln\right) \\
        &\leq \left(\frac{\Tilde{c}n^2}{n^2}\right)^{1/2} \left(\frac{(L_c+\lambda n)}{Ln}\right)^{1/2} 
    \end{align*}
    where the first inequality is due to Hölder's inequality, the second inequality comes from the comparison with~$J_n$ and the final inequality follows by Markov's inequality. This can be smaller than~$\epsilon$ by choosing~$L$ large enough in terms of the constants~$\Tilde{c}$ and~$\lambda$ but independent of~$n$. 
    
    Therefore~$r_n/n\rightarrow \mu_{\mathrm SP}$ as~$n\rightarrow \infty$ in~$L^1$ considering the almost sure convergence we have proved before. It remains to show the same convergence across all times, i.e., for~$r_t/t$. To do that, we would have to show that for any~$\epsilon>0$ there exists~$t_0\in[0,\infty)$ large enough so that
    \begin{equation} \label{eq_target_L1}
        {\mathbb E}_{\mathcal S} \left[\frac{r_t}{t}-\mu_{\mathrm{SP}}\right]<\epsilon \quad \text{ for all }t\geq t_0.
    \end{equation}
    
    But we have that~\eqref{eq_target_L1} is bounded above by:
    \begin{equation}
    \underbrace{{\mathbb E}_{\mathcal S} \left[\left|\frac{r_t}{t}-\frac{r_{\lfloor t\rfloor}}{t}\right|\right]}_{(1)} + \underbrace{{\mathbb E}_{\mathcal S} \left[\left|\frac{r_{\lfloor t\rfloor}}{t} - \frac{r_{\lfloor t \rfloor}}{t}\right|\right]}_{(2)}
         + \underbrace{{\mathbb E}_{\mathcal S}\left[\frac{r_{\lfloor t \rfloor}}{t}- \mu_{\mathrm{SP}}\right]}_{(3)}
    \end{equation}
    and each of those terms can be made smaller than~$\epsilon /3 $ by choosing~$t$ sufficiently large. Indeed, the term in~$(1)$ is bounded by~$\frac{1}{t}{\mathbb E}_{\mathcal S}[r_t-r_{\lfloor t \rfloor}]\leq \frac{\lambda}{t}$; the term in~$(2)$ is bounded by~$\frac{{\mathbb E}_{\mathcal S}[r_{\lfloor t\rfloor}]}{t\lfloor t \rfloor}\leq \frac{\lambda t}{t\lfloor t \rfloor}$; finally, the term in~$(3)$ can be made smaller than~$\frac{\epsilon}{3}$ choosing~$t$ sufficiently large because~$r_n/n\rightarrow \mu_{\mathrm{SP}}$ in~$L_1$ as~$n\rightarrow \infty$. 
\end{proof}

Finally, we prove Theorem~\ref{thm_Spont_CLT} and~\ref{thm_CLT_IS} using a strategy very similar to the one made in \cite{mountford2016functional}. The idea is that, given a time $t \in [0, \infty)$, we decompose the value of $r_t$ by comparing it to the rightmost particle at the time of the last renewal before $t$. We also identify which renewal time this is by comparing it to its expected value.

\begin{proof}[Proof of Theorem~\ref{thm_Spont_CLT} and~\ref{thm_CLT_IS}]
    We have the following decomposition:
    \begin{equation*}
    \frac{r_t - \mu_{\mathrm{SP}}t}{\sqrt{t}} = 
    \underbrace{\frac{r_t - r_{\sigma(t)}}{\sqrt{t}}}_{\text{(1)}} + 
    \underbrace{\frac{r_{\sigma(t)} - r_{\sigma_{\lfloor t/\mu_T\rfloor}}}{\sqrt{t}}}_{\text{(2)}} + 
    \underbrace{\frac{r_{\sigma_{\lfloor t/\mu_T\rfloor}} - r_{\sigma_0}}{\sqrt{t}}}_{\text{(3)}} + 
    \underbrace{\frac{r_{\sigma_0}}{\sqrt{t}}}_{\text{(4)}}
\end{equation*}

We will show that the terms in~$(1)$,~$(2)$ and~$(4)$ go to zero in probability, and that the term in~$(3)$ converges in distribution to~$\mathcal{N}(0,\sigma^2_{\mathrm{SP}})$.

The fact that~$(4)$ converges to zero in probability is clear since~$r_{\sigma_0}$ is a finite random variable since~$\sigma_0$ is finite almost surely. Note that the numerator in~$(3)$ is just the sum of~$\lfloor t/\mu_T\rfloor$ i.i.d. terms of mean~$\mu_{S}$ because of Proposition~\ref{iidSequence}, so it converges in distribution to~$\mathcal N(0,\sigma_{\mathrm{SP}}^2)$ because of the central limit theorem. 

We work with the term in~$(1)$ first. Let~$\delta>0$ be given and consider the good event~$G:=\left\{\mid\frac{N(t)}{t}-\frac{1}{\mu_T}\mid\leq \delta\right\}$. For any given~$\epsilon>0$, we have that: 
\begin{equation} \label{term1Eq}
    \bar{\mathbb{P}}_\chi^x\left(\frac{|r_t - r_{\sigma_{N(t)}}|}{\sqrt{t}}>\epsilon\right)\leq \bar{\mathbb P}_\chi^x\left(\left\{|r_t-r_{\sigma_{N(t)}}|>\epsilon\sqrt{t}\right\}\cap G\right) + \bar{\mathbb P}_\chi^x(G^c)
\end{equation}

The second term in~\eqref{term1Eq} goes to zero as~$t\rightarrow\infty$ due to the renewal theorem. To deal with the first term, note that if we let~$n_1 = \lfloor t/\mu_T-\delta t\rfloor $ and~$n_2 = \lceil t/\mu_T+\delta t\rceil$, we can bound it by:
\begin{equation*}
    \bar{\mathbb P}_\chi^x\left(\bigcup_{n=n_1}^{n_2}\sup\{|r_s-r_{\sigma_n}|\,:\,s\in [\sigma_n,\sigma_{n+1})\}>\epsilon\sqrt{t}\right) \leq (2\delta t + 1)Ce^{-c(\epsilon\sqrt{t})^p}
\end{equation*}
where the inequality follows from a union bound along with Lemma~\ref{lem_timeSpaceRenewalBeforet}. 

To deal with~$(2)$, we proceed as follows. First, we rewrite the good event~$G$ as~$G = G_1\cup G_2$ with~$G_1 = \{N_t\in [\frac{t}{\mu}-\delta t,\frac{t}{\mu}]\}$ and~$G_2 = \{N_t\in [\frac{t}{\mu},\frac{t}{\mu}] \}$. Then, the following holds:
\begin{align}
        &\bar{\mathbb{P}}_{\chi}^x\left(\frac{r_{\sigma_{N(t)}}-r_{\sigma_{\lfloor t/\mu\rfloor}}}{\sqrt{t}}>\epsilon\right)  \leq \\ \label{threeTerms}
        &{\bar{\mathbb{P}}_{\chi}^x} \left(G^c\right) + \bar{\mathbb{P}}_{\chi}^x\left(\left\{\frac{r_{\sigma_{N(t)}}-r_{\sigma_{\lfloor t/\mu\rfloor}}}{\sqrt{t}}>\epsilon \right\} \cap G_1\right) + \bar{\mathbb{P}}_{\chi}^x\left(\left\{\frac{r_{\sigma_{N(t)}}-r_{\sigma_{\lfloor t/\mu\rfloor}}}{\sqrt{t}}>\epsilon\right\} \cap G_2\right)
\end{align}

The first term goes to 0 once more due to the renewal theorem. The second and third terms are bounded in a similar way, and therefore we will only prove the bound for the third one. We have that  
\begin{align*}
    \bar{\mathbb{P}}_{\chi}^x \left(\frac{r_{\sigma_{N(t)}}-r_{\sigma_{\lfloor t/\mu\rfloor}}}{\sqrt{t}}>\epsilon \cap G_2\right)  
    &\leq \bar{\mathbb{P}}_{\chi}^x \left(\{\max_{\lfloor\frac{t}{\mu}\leq i \leq \lfloor\frac{t}{\mu}+\delta t\rfloor}\frac{r_{\sigma_{i}}-r_{\sigma_{\lfloor t/\mu\rfloor}}}{\sqrt{t}}>\epsilon\} \cap G_2\right) \\
    & \leq \bar{\mathbb{P}}_{\chi}^x\left(\max_{\lfloor\frac{t}{\mu}\leq i \leq \lfloor\frac{t}{\mu}+\delta t\rfloor}|r_{\tau_{i}}-r_{\tau_{\lfloor t/\mu\rfloor}}| >\epsilon{\sqrt{t}}\right) \\
    &\leq \delta t \bar{\mathbb{P}}_{\chi}^x\left(|r_{\sigma_1}|>\epsilon\sqrt{t}\right) \leq \delta t\frac{\text{Var}(r_{\sigma_1})}{\epsilon^2 t}
\end{align*}

where the third inequality follows from a union bound and the fourth inequality follows from Kolmogorov's inequality; the final term on the above equation can be made arbitrarily small by taking $\delta$ sufficiently small. Note that $\text{Var}(r_{\sigma_1})$ is finite since $\bar{\mathbb{P}}_{\chi}^x(\sigma_1>t)\leq Ce^{-ct^p}$.
\end{proof}

\section{Detailed constructions of the processes}\label{appendix_construction}

\paragraph{}The main goal of this section is to obtain Lemma~\ref{ActivePathIFF}; we do it in a very careful way in order to avoid any circularities in the construction of those processes.

 \begin{definition}
    \textbf{(Infection path)} Consider a given a graphical construction\\
    $\mathcal{H}= \Big((N_1^{x,y})_{x\sim y}, (N_2^{x,y})_{x\sim y}, (U^x)\Big)$ and a time interval $I\subset [0,\infty)$. We say that $\gamma:I\rightarrow\mathbb{Z}$ is an infection path if $\gamma$ is a càdlàg function such that the following two properties are satisfied:

    \begin{itemize}
        \item \textbf{(P1)}: for all $s\in I$, $s\notin U^{\gamma(s)}$
        \item \textbf{(P2)}: if $\gamma(s-)\neq \gamma (s)$, then $\gamma(s)\in N^{\gamma(s-), \gamma(s)}_1$ for all $s\in I$
    \end{itemize}
    
     If there exists such $\gamma:I\rightarrow \mathbb{Z}$ with $x=\gamma(s)$ and $y=\gamma(t)$ for $s\leq t$ where $[s,t]\subset I$, we write either $(x,s)\rightarrow(y,t)$ or we write $(x,s)\overset{\mathcal{H}}{\rightarrow} (y,t)$ in case we aim to stress the graphical construction used. If $I$ is of the form $[t,\infty)$ for some $t\geq 0$ with $\gamma(t)=x$, we write $(x,t)\rightarrow \infty$. 
\end{definition}

From a graphical construction $\mathcal{H}$ and an initial configuration in $\{-1,0,1\}^{\mathbb Z}$, one can be tempted to construct either a Spont process or a contact process with inherited sterility declaring a site $x$ to be occupied at time $t$ if and only if there exists an infection path starting from an occupied site at time $0$ that reaches $x$ at time $t$ and avoids any sites of type $-1$. The problem with this is that then the construction of the process becomes circular: to construct the process and claim that a site is occupied if only if there exists such infection path, we find the problem that the existence of this path depends upon determining the states of the site this path goes through, which in turn requires the process to be constructed. We fix this by constructing those process as a limit of truncated processes of the same type.

\subsection{Graphical construction of the Spont process} \label{constructionSPON}

\paragraph{} For a given~$n\in\mathbb{N}$ and a given graphical construction~$\mathcal{H}$, let~$\mathbb{Z}_n=\mathbb{Z}\cap[-n,n]$ and let~$\mathcal{H}_n$ be the restriction of~$\mathcal{H}$ to the box~$\mathbb{Z}_n$, i.e., $\mathcal{H}_n = \Big((N_1^{x,y})_{x\sim y,x,y\in\mathbb{Z}^n}, (N_2^{x,y})_{x\sim y,x,y\in\mathbb{Z}^n}, (U^x)_{x\in\mathbb{Z}^n}\Big)$ where the processes $N^1_{x,y},N^2_{x,y}$ and $U^x$ belong to $\mathcal{H}$. 

Let $\xi_0\in\{-1,0,1\}^{\mathbb{Z}}$ be an initial configuration. Let $\xi_0^n$ be the restriction of this configuration to the box $\mathbb{Z}^n$, i.e., $\xi_0^n(x) = \xi_0(x)\mathbbm{1}_{\{x\in\mathbb{Z}_n\}}$. We will construct a Spont process in a box $(\xi_t^n)_{t\geq 0}$ taking values in $\{-1,0,1\}^{\mathbb{Z}_n}$ via a sequence of stopping times in the following way. 

Let $\sigma_0=0$ and suppose that we have constructed the process until some moment $\sigma_k$. Define $\sigma_{k+1}=\inf\{t>\sigma_k\,:\,\mathcal{H}_n\text{ has an arrival}\}$ (note that the inequality $\sigma_{k+1}>\sigma_k$ is strict). Let $\xi_t^n=\xi_{\sigma_k}^n$ for $t\in[\sigma_k,\sigma_{k+1})$. For the updated configuration at time $\sigma_{k+1}$, let $\xi_{\sigma_k+1}^n$ be given in the following way depending on which arrival we had at time $\sigma_{k+1}$:

\begin{itemize}
    \item If $\sigma_{k+1}\in N_{x,y}^1$ and if $\xi_{\sigma_k}^n(x)=1$ and $\xi_{\sigma_k}^n(y)=0$, let $\xi_{\sigma_{k+1}}^n(y) =1$ and $\xi_{\sigma_{k+1}}^n(z) = \xi_{\sigma_k}^n(z)$ if $z\neq y$. 
    \item If $\sigma_{k+1}^n\in N_{x,y}^2$ and if  $\xi_{\sigma_k}^n(y)=0$, let $\xi^n_{\sigma_{k+1}}(y) =-1$ and $\xi_{\sigma_{k+1}}^n(z) =\xi^n_{\sigma_k}(z)$ if $z\neq y$. 
    \item If $\sigma^n_{k+1}\in U^x$, let $\xi^n_{\sigma_{k+1}}(x)=0$ and $ \xi^n_{\sigma_{k+1}}(z) =\xi^n_{\sigma_k}(z) $ if $z\neq x$. 
\end{itemize}

Note that those three possibilities are mutually exclusive since the probability of two independent Poisson point processes having the same arrival time is zero.

\begin{claim} \label{ConstructionAttractiveClaim}
    Let $n,m\in\mathbb{N}$ with $n\leq m$ be given. It follows that $\xi^n_t(x)\leq\xi_t^m(x)$ for all $x\in\mathbb{Z}_n$ and for all $t\in[0,\infty)$.
\end{claim}

\begin{proof}
    By contradiction, assume that there exists a first problematic moment $P>0$ and a problematic site $x_P\in \mathbb{Z}_n$ where $\xi^n_P(x_P) >\xi^m_P(x_P)$ for some problematic site $x_P\in\mathbb{Z}_n$. This problematic moment corresponds to a moment of arrival of $\mathcal{H}_n$ that involves a change at site $x_P\in \mathbb{Z}_n$. If $\xi^n_P(x_P)=1$, this implies that there exists an infection path from some site in $\mathbb{Z}_n$ at time $0$ that reaches $x_P$ at time $P$ and this infection path always jumps to sites that are not blocked; by minimality of $P$, the existence of such path would also imply that $\xi^m_P(x_P)=1$. If $\xi_P^n(x_P)=0$, the healing mark at site $x_P$ at moment $P$ would also make this site empty for the process $(\xi^m_t)_{t\geq 0}$, and thus $\xi_P^m(x_P)=0$ as well. 
\end{proof}

\begin{claim} \label{SpontUniformLimitBox}
    Let~$N\in \mathbb N$ and~$T\in[0,\infty)$ be given. Then, there exists~$M\in \mathbb N$ with~$M=M(\mathcal{H},N,t)\geq N$ such that for any~$(x,t)\in[-N,N]\times[0,T]$ one has that $\xi_t^n(x)=\xi_t^m(x)$ for all $n,m\geq M$. 
\end{claim}

\begin{proof}
    Note that we ask for $M\geq N\geq|x|$ so that both truncated processes $(\xi_t^n)_{t\geq 0}$ and $(\xi_t^m)_{t\geq 0}$ are well-defined in $(x,t)$ for any~$n,m\geq M$. 
    
    The probability that there exists an infection path $(y,0)\overset{\mathcal{H}}{\rightarrow} (x,t)$ with $|x-y|>k$ is smaller or equal that the probability of a Poisson process of rate $\lambda p$ to have more than $k$ realisations before $t$, which goes to zero as $k\rightarrow\infty$. 

    Thus, let $N'=N(\mathcal H, N,T)$ be such that for all $(x,t)\in[-N,N]\times[0,T]$ there is not $(y,s)$ such that $(y,s)\overset{\mathcal{H}}{\rightarrow}(x,t)$ with $|y|>N'$ and $s\leq t$. Let also $M=M(\mathcal H,N,T)$ be such that for all $(y,s)\in[-N',N']\times[0,T]$ there is not $(z,u)$ such that $(z,u)\overset{\mathcal{H}}{\rightarrow}(x,t)$ with $|z|>M$ and $u\leq s$. Let $n,m\geq M$ be given. 

    Assume, by contradiction, that there exists the first problematic moment $P$ where the property we aim to verify fails at some site $x_P$, i.e., $\xi_P^n(x_P)\neq \xi_P^m(x_P)$ for some $P\in[0,T]$ and $x_P\in[-N,N]$. If the arrival time in $(x_P,P)$  is a blocking or healing mark, it would affect both processes in the same way. If this arrival corresponds to a transmission arrow, we assume without loss of generality that $\xi^n_P(x_P)=1$ but $\xi^m_P(x_P)\neq 1$. This would imply the existence of an infection path $\gamma_P:[0,P]\rightarrow \mathbb Z$ such that $\gamma_P(s)\in[-N',N']$ for all $s\in[0,P]$ (by the choice of $N'$). If $n<m$, then $\xi_s^ n(\gamma(s))=1\leq \xi_s^m(\gamma(s))$ for all $s\in[0,P]$ because of Claim \ref{ConstructionAttractiveClaim} and thus $\xi^m_P(x_P)=1$. Suppose now that $n>m$. 

    First, let $M\in \mathbb N$ be given and let $\gamma_R^M:[0,T]\rightarrow\mathbb Z$ be the path (not necessarily an infection one) that starts from $M+1$ and jumps to the left as soon as it finds a transmission arrow to the left, namely: let $J_1 = \inf\{s\geq 0\,:\,s\in N_1^{M+1,M}\}$ and let $\gamma_R^M(t) = M+1$ for $t\in[0,J_1)$ and $\gamma_R^M(J_1)=M$. Suppose that we have defined this function until some time $J_k$. Then, let $J_{k+1} = \inf\{s\geq 0\,:\,s\in N_1^{M-(k+1),M-k}\}$ and let $\gamma_R^M(t) = M-(k+1)$ for $t\in [J_k,J_{k+1})$ and $\gamma_R^M(J_{k+1}) = M-k$. Consider this function built in that way until time $T$. Consider its analogue $\gamma_L^M:[0,T]\rightarrow \mathbb Z$ the path starting from $-M-1$ that jumps to the right as soon as it finds a transmission arrow to the right, constructed in the same way. Increase $M$ so that $\gamma_R^M(T)>N'$ and $\gamma_L^M(T)<-N'$. 

    Let $D_1=\inf\{t\geq 0\,:\,\exists x\in[-M,M]\text{ s.t. }\xi^n_t(x)\neq \xi^m_t(x)\}$. Let $x_{D_1}$ be such a problematic site. It must be the case that there was an arrival of $\mathcal H$ affecting that site, otherwise, processes would have continued to agree there. If this arrival was due to a healing mark or a blocking mark, the processes would still agree on that site because of similar arguments that we have discussed before. The processes would still agree on that site if the change at $\mathcal H$ happened due to a transmission arrow with origin site in $(-M,M)$. Therefore, the only way is if $x_{D_1}=M$ or $x_{D_1}=-M$ and that site was occupied in one process but not in the other due to an arrival from a transmission arrow originated from outside the box $[-M,M]$. 

    Assume that the sequence $D_k$ is defined for some $k\in \mathbb N$. Let $D_{k+1}=\inf\{t\geq 0\,:\,\exists \in[-M+k,M-k]\text{ s.t. }\xi^n_t(x)\neq \xi^m_t(x)\}$. The exact same conclusion that we had for the time $D_1$ holds for the time $D_{k+1}$, i.e., the change was cause by a transmission arrow arriving from outside the box $[-M+k,M-k]$ to one of its boundary points with the additional information that we also have $D_k<D_{k+1}$. Let $k^*=\inf\{k\in \mathbb N\,:\,D_k>T\}$. By definition, before $D_{k^{*}}$, processes agreed on the region $[-M+k^*,M-k^*]$. But note that due to the existence of the paths $\gamma_R^M$ and $\gamma_L^M$, one has $-M+k^*<-N'$ and $M-k^*>N'$. 
    
    Since $P<D_{k^*}$ because $P\leq T$ and since $\gamma_P(s)\in[-N',N']$ and processes agree every space-time point visited by path $\gamma_P$, and thus this path would have implied that also $\xi^m_P(x_P)=1$, leading to a contradiction because of the existence of such a problematic moment. 
\end{proof}

Because of Claim \ref{SpontUniformLimitBox}, we can let $(\xi_t)_{t\geq 0}$ be the Spont process  started from $\xi_0$ constructed with $\mathcal{H}$ defined by the (well-defined) limit for any $(x,t)\in\mathbb{Z}\times[0,\infty)$: 

\begin{equation}
    \xi_t(x) = \lim_{n\rightarrow\infty}\xi_t^n(x).
\end{equation}

\subsection{Graphical construction of the IS process} \label{constructionCPIS}

\paragraph{}We also construct the CPIS process via a limit of truncated processes. The main difference from the truncated argument made in Subsection \ref{constructionSPON} is that we don't have to truncate the graphical construction too and study the process in a box as sites of type~$-1$ cannot appear everywhere, only at sites neighbouring sites of type~$1$.

Let~$\eta_0\in\{-1,0,1\}^{\mathbb{Z}}$ be a given initial configuration and~$\mathcal{H}$ be a graphical construction. For a given~$n\in\mathbb{N}$, consider~$\eta_0^n$ the truncated initial configuration~$\eta_0^n(x) = \eta_0(x)\mathbbm{1}_{\{x\in\mathbb{Z}_n\}}$. Using~$\mathcal{H}$, we will construct~$(\eta_t^n)_{t\geq 0}$ a contact process with inherited sterility started from~$\eta_0^n$ in the following way. 

Let~$\sigma_0=0$ and suppose that the process has been built until some time~$\sigma_k$. Define~$\sigma_{k+1}=\inf\{t\geq \sigma_k\,:\,\mathcal{H}_{n+k}\text{ has an arrival}\}$. Note that we are slowly increasing the region where the arrival happens by one unit at every time increment. Again, as before, one has that~$\sigma_{k+1}$ is strictly bigger than~$\sigma_k$. Then, let~$\eta^n_t = \eta_{\sigma_k}^n$ for~$t\in[\sigma_k,\sigma_{k+1})$. We will now build~$\eta_{\sigma_{k+1}}^n$ depending on which arrival we had at moment~$\sigma_{k+1}$ (and note that those are mutually exclusive since the chances of independent Poisson point processes having the same arrival are zero):

\begin{itemize}
    \item If~$\sigma_{k+1}\in N^1_{x,y}$ and if~$\eta^n_{\sigma_k}(x)=1$ and~$\eta^n_{\sigma_k}(y)=0$, let~$\eta_{\sigma_{k+1}}^n(y)=1$ and ~$\eta_{\sigma_{k+1}}^n(z)=\eta_{\sigma_{k}}^n(z)$ for all~$z\neq y$.
    \item If~$\sigma_{k+1}\in N^2_{x,y}$ and if~$\eta^n_{\sigma_k}(x)=1$ and~$\eta^n_{\sigma_k}(y)=0$, let~$\eta_{\sigma_{k+1}}^n(y)=-1$ and let ~$\eta_{\sigma_{k+1}}^n(z)=\eta_{\sigma_{k}}^n(z)$ for all~$z\neq y$.
    \item If~$\sigma_{k+1}\in U^{x}$, then let~$\eta_{\sigma_{k+1}}^n(x)=0$ and let ~$\eta_{\sigma_{k+1}}^n(z)=\eta_{\sigma_{k}}^n(z)$ for all~$z\neq x$.
\end{itemize}

\begin{claim} \label{CPISActivePAthClaim}
    Let~$N\in\mathbb{N}$ and~$T\in[0,\infty)$ be given. Then, there exists~$M\in\mathbb{N}$ with $M=M(\mathcal{H},N,T)\geq N$ such that for any~$(x,t)$ in the space-time box~$[-N,N]\times[0,T]$ we have~$\eta_t^n(x)=\eta_t^m(x)$ for all~$n,m\geq M$. 
\end{claim}

\begin{proof}
    Let~$N\in\mathbb{N}$ and~$T\in[0,\infty)$ be given. Let~$N'\in\mathbb{N}$ be large enough so that for all~$y\in\mathbb Z$ with~$|y|>N'$ and for all~$s\in[0,T]$ one has that~$(y,s)\overset{\mathcal{H}}{\nrightarrow}(x,t)$ for any~$(x,t)\in[-N,N]\times[0,T]$. Now, for the space-time box~$[-N',N']\times[0,T]$, consider~$M\in\mathbb N$ large enough so that for all~$z\in\mathbb{Z}$ with~$|z|>M$ and for all~$s\in[0,T]$, one has that~$(z,s)\overset{\mathcal{H}}{\nrightarrow}(x,t)$ for any~$(x,t)\in[-N',N']\times[0,T]$ with $t\geq s\geq 0$. 

    Let~$n,m\geq M$ be given; assume that~$P\in[0,T]$ is the first problematic moment where the property we are looking for is not verified, i.e.,~$P$ is the first moment for which~$\eta^n_t(y)\neq \eta_t^m(y)$ for some~$y^*\in[-N,N]$ and for some~$t\in[0,T]$. Let~$y^*$ be the site in~$[-N,N]$ where this property was first broken at moment~$P$. Note that~$P$ has to be an arrival moment for~$\mathcal{H}_{n+1}$: otherwise the processes in~$[-N,N]$ would continue to be the same as the moment immediately before and the processes would then still agree on that region. 

    Since at moment~$P-$ we have~$\eta_{P-}^n(y)=\eta_{P-}^m(y)$ for all~$y\in[-N,N]$, it is easy to see that if the arrival of~$\mathcal{H}_{n+1}$ happened involving sites in the interior of~$[-N,N]$, then the processes would still agree after this arrival which would contradict the existence of~$P$. Therefore, the change must involve either an arrival of~$N^{x^*,y^*}_1$ or~$N^{x^*,y^*}_2$ where~$(x^*,y^*)=(N+1,N)$ or~$(x^*,y^*)=(-N-1,-N)$ and the processes must be such that~$x^*$ was occupied in one of the processes but not on the other.

    Suppose without loss of generality that~$\eta^n_{P-}(x^*)=1$; we want to conclude that also~$\eta^m_{P-}(x^*)=1$. Indeed, if~$\eta^n_{P-}(x^*)=1$, then there exists~$\gamma:[0,P]\rightarrow\mathbb{Z}$ active infection path for~$\eta^n_t$ such that~$\eta^n_0(\gamma(0))=1$,~$\gamma(P)=x^*$ and~$\gamma(t)\in[-N',N']$ for all~$t\in[0,P]$ by the choice of~$N'$. This path~$\gamma$ would not be active if any of the visited sites suffered the influence of a sterile individual that came from a fertile site say at position~$P^*$ that was present at the initial moment in~$\eta^m_0$ but not on~$\eta^n_0$; but then this site is such that~$|P^*|>M$ since~$n,m\geq M$. But the existence of such a site contradicts the choice of~$M$. 
\end{proof}

Therefore, we can define the contact process with inherited sterility constructed with~$\mathcal{H}$ started from~$\eta_0$ as the (well-defined) limiting process for any~$x\in\mathbb Z$ and any~$t\in[0,\infty)$:

\begin{equation}
    \eta_t(x) = \lim_{n\rightarrow\infty}\eta^n_t(x)
\end{equation}

\paragraph{}We are finally ready to prove Lemma~\ref{ActivePathIFF} that enable processes to be defined using paths.

\begin{proof}[Proof of Lemma~\ref{ActivePathIFF}]
    Let~$(x,t)\in\mathbb Z\times [0,\infty)$ and~$\mathcal{H}$ a graphical construction be given. 
    Then, there exists~$N\in\mathbb{N}$ such that~$(y,s)\overset{\mathcal{H}}{\nrightarrow} (x,t)$ for any~$y$ with~$|x-y|>N$ and for any~$s\in[0,t]$. Let~$N^* = |x|+N$. Consider the space-time box~$[-N^*,N^*]\times[0,t]$. 

    Because of Claims~\ref{SpontUniformLimitBox} and~\ref{CPISActivePAthClaim}, there exists~$M\in\mathbb{N}$ such that~$\chi^{M+1}_s(z) = \chi_s(z)$ for all~$(z,s)\in[-N^*,N^*]\times[0,t]$ (and also~$\chi^{M+k}_s(z) =\chi^{M+1}_s(z)$ for any~$k\in\mathbb N$). In particular,~$\chi^{M+1}_t(x)=1$ and by construction of the truncated process, there exists~$\gamma:[0,t]\rightarrow\mathbb{Z}$ satisfying:
    \begin{itemize}
        \item $\beta_0(\gamma(0))=1$, so that~$\gamma$ starts from an occupied site;
        \item ~$\gamma(0)\in[-N,N]$,otherwise this would contradict the choice of~$N$;
        \item ~$\gamma(t) =x$, so that the infection is carried until~$x$;
        \item ~$\gamma(s)\in[-N^*,N^*]$ for all~$s\in[0,t]$, otherwise, since~$N^*>N$, this would also contradict the choice of~$N$. 
        \item ~$\chi^{M+1}_s(\gamma(s))=1$ for all~$s\in[0,t]$
    \end{itemize}

    Finally, since~$\chi^{M+1}_s(\gamma(s)) = \chi_s(\gamma(s))$ for all~$s\in[0,t]$, this proves the existence of an active infection path as the one claimed. The converse is clearly true. 
\end{proof}

\section{Open questions}
We would like to conclude this work with a few open questions that we find interesting and may be the subject of future work. 
\begin{itemize}
    \item In this work, we have only proven the linearity of the inherited contact process in a supercritical regime for the Spont process. However, we believe that the supercritical region is more important for IS and that the linearity result is true in a wider range of parameters (and perhaps throughout the entire supercritical phase for the IS process). We can observe this phenomenon in the simulations below: in Figure~\ref{fig:sp_sub_critical}, we used parameters that appear to be subcritical for Spont and we observe that IS still survives and still exhibits linear growth.
    \item In the parameter regime in which all processes survive, we may wish to compare the different speeds. We see in Figure~\ref{fig:the_three} that the Spont process appears to be much slower than the other two processes. 
    \item We may also be interested in the monotonicity of the velocity with respect to the parameters. We know that the Spont process is monotonous with respect to $p$ and that IS is not and both processes are not monotonous with respect to $\lambda $, in the sense of the existence of a coupling (cf \cite{velasco2024} for details).  But what about the overall behaviour embodied by speed? 
    \item Finally, our work is really one-dimensional dependent but in higher dimensions, an asymptotic shape seems to appear as can be seen in Figure~\ref{fig:fig_2d}. 
    \end{itemize}
    \begin{center}

\begin{figure}[h!]
    \centering

    \begin{subfigure}{0.34\textwidth}
        \centering
        \includegraphics[width=\linewidth, angle=180]{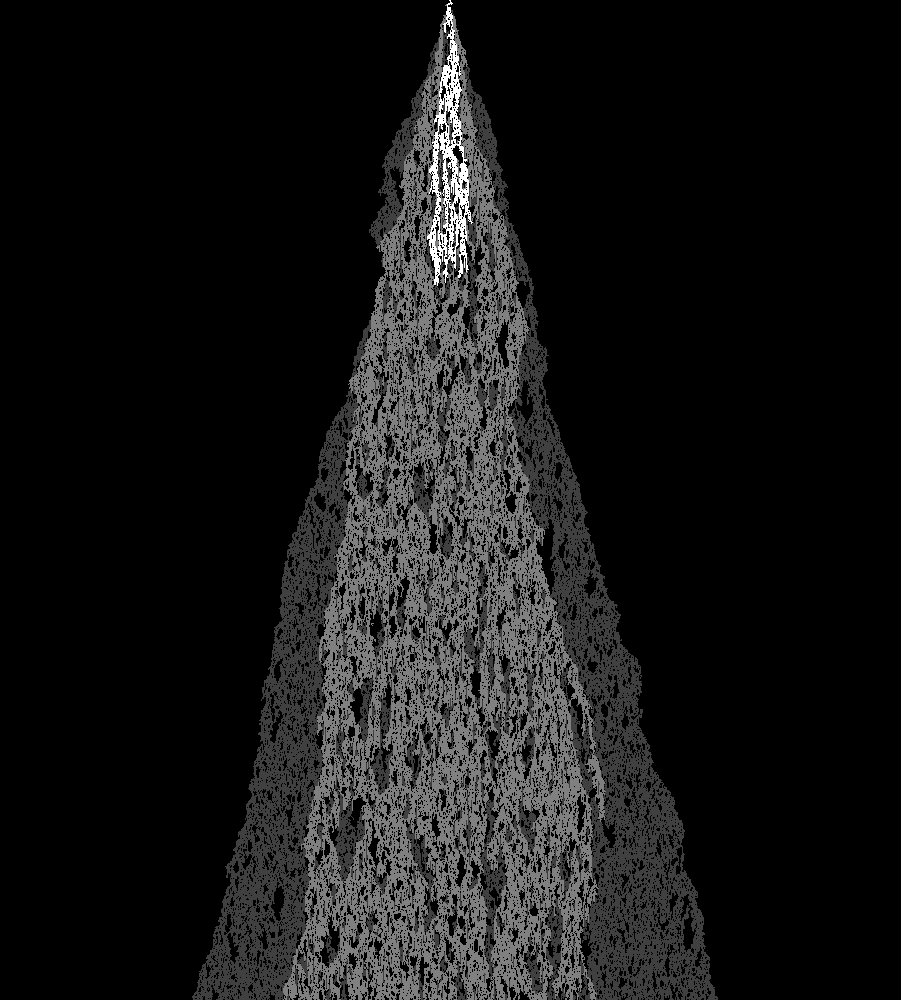}
        \caption{$\lambda = 2.77$, $p = 0.9$}
        \label{fig:sp_sub_critical}
    \end{subfigure}
    \hspace{0.03\textwidth} 
    \begin{subfigure}{0.29\textwidth}
        \centering
        \includegraphics[width=1.5\linewidth, height=1.3\linewidth, angle=180]{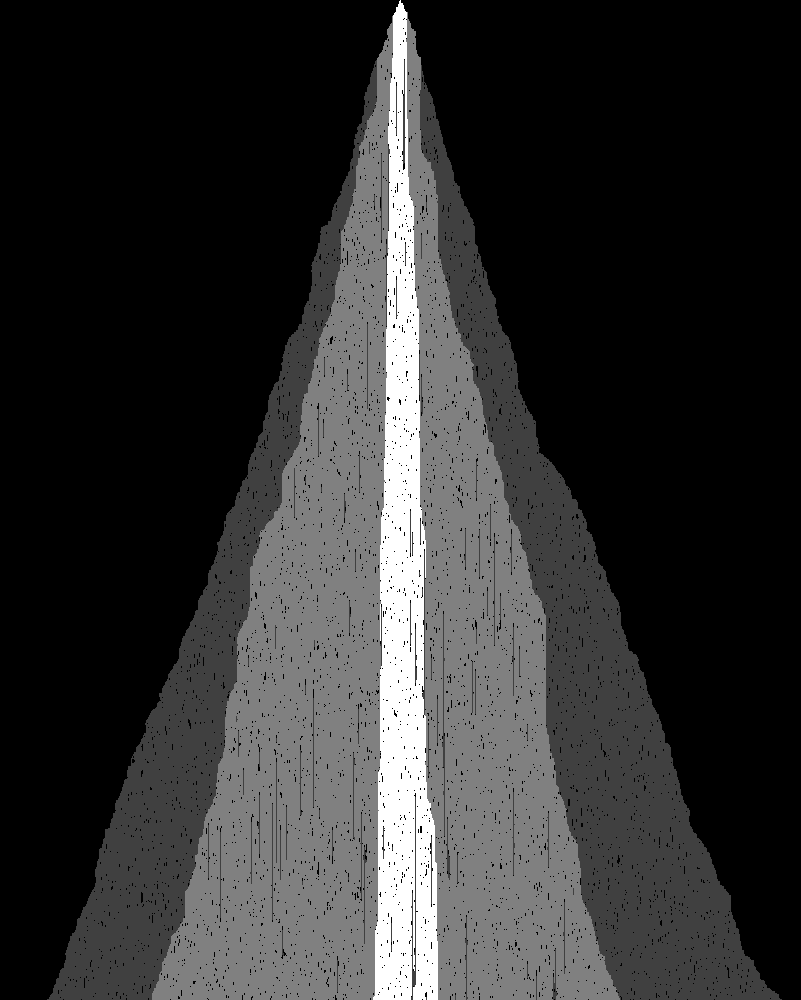}
        \caption{$\lambda = 20$, $p = 0.975$}
        \label{fig:the_three}
    \end{subfigure}

    \vspace{0.6em}
   
    \begin{subfigure}{0.4\textwidth}
        \includegraphics[width=\linewidth]{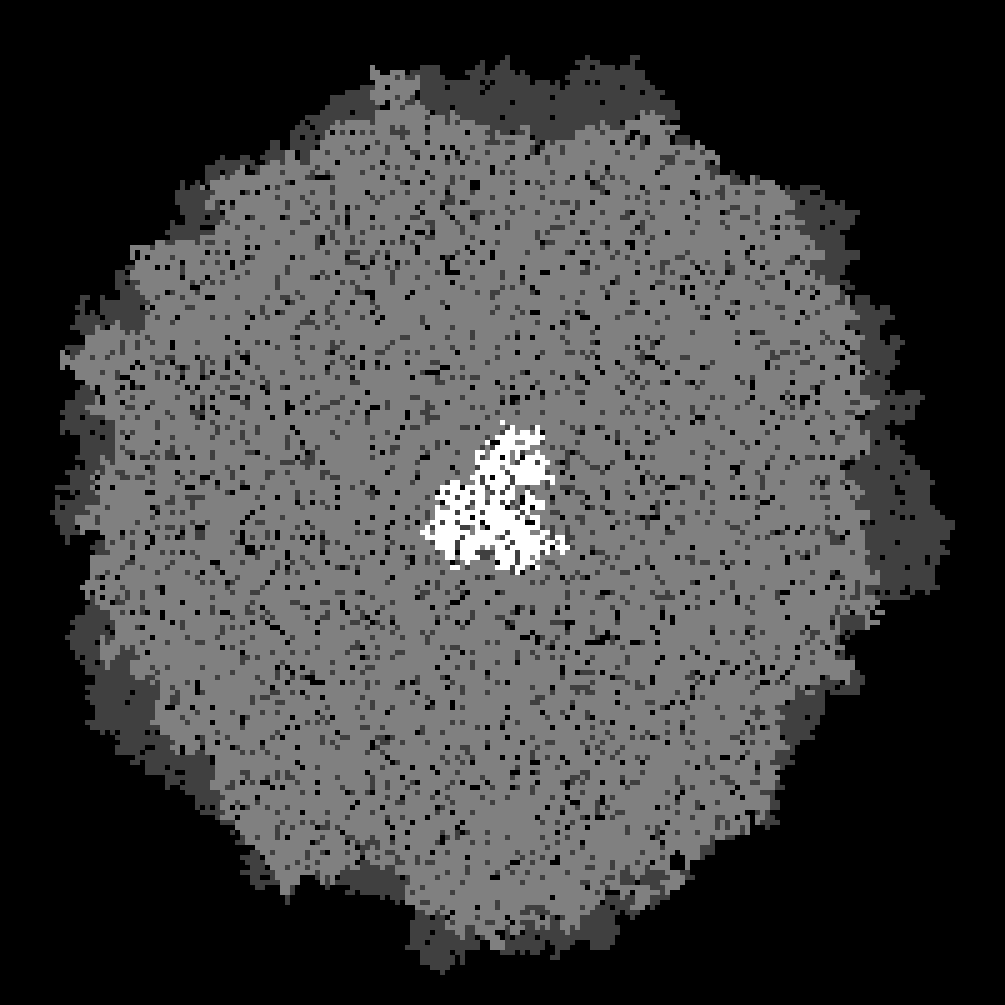}
        \caption{$d = 2$, $\lambda = 5.5$, $p = 0.9$, fixed time}
        \label{fig:fig_2d}
    \end{subfigure}

    \caption{Coupled processes. (a)–(b): one-dimensional cases. (c): two-dimensional case.}
    \label{fig:placeholder}
\end{figure}

    \end{center}


\newpage 

@article{harris1974contact,
  title={Contact interactions on a lattice},
  author={Harris, Theodore E.},
  journal={The Annals of Probability},
  volume={2},
  number={6},
  pages={969--988},
  year={1974},
  publisher={Institute of Mathematical Statistics}
}

@article{velasco2024,
  title={Extinction and survival in inherited sterility},
  author={Velasco, Sonia},
  journal={arXiv preprint arXiv:2404.11963},
  year={2024},
  note={\url{https://arxiv.org/abs/2404.11963}}
}

@article{kuczek1989central,
  title={The central limit theorem for the right edge of supercritical oriented percolation},
  author={Kuczek, Thomas},
  journal={The Annals of Probability},
  volume={17},
  number={4},
  pages={1322--1332},
  year={1989},
  publisher={Institute of Mathematical Statistics},
  doi={10.1214/aop/1176991237}
}

@article{mountford2016functional,
  title={Functional central limit theorem for the interface of the symmetric multitype contact process},
  author={Mountford, Thomas and Valesin, Daniel},
  journal={ALEA Latin American Journal of Probability and Mathematical Statistics},
  volume={13},
  number={1},
  pages={481--519},
  year={2016},
  doi={10.30757/alea.v13-20}
}

@article{durrett1984oriented,
  title={Oriented percolation in two dimensions},
  author={Durrett, Richard},
  journal={The Annals of Probability},
  volume={12},
  number={4},
  pages={999--1040},
  year={1984},
  publisher={Institute of Mathematical Statistics}
}

@book{billingsley2013convergence,
  title={Convergence of probability measures},
  author={Billingsley, Patrick},
  year={2013},
  publisher={John Wiley \& Sons}
}

@article{remenik2008contact,
  title={The contact process in a dynamic random environment},
  author={Remenik, Daniel},
  journal={The Annals of Applied Probability},
  volume={18},
  number={6},
  pages={2392--2420},
  year={2008},
  doi={10.1214/08-AAP525}
}

@article{broman2007stochastic,
  title={Stochastic domination for a hidden Markov chain with applications to the contact process in a randomly evolving environment},
  author={Broman, Erik I.},
  journal={The Annals of Probability},
  volume={35},
  number={6},
  pages={2263--2293},
  year={2007},
  doi={10.1214/0091179606000001187}
}

@article {steif2007critical,
    AUTHOR = {Steif, Jeffrey E. and Warfheimer, Marcus},
     TITLE = {The critical contact process in a randomly evolving
              environment dies out},
   JOURNAL = {ALEA Lat. Am. J. Probab. Math. Stat.},
  FJOURNAL = {ALEA. Latin American Journal of Probability and Mathematical
              Statistics},
    VOLUME = {4},
      YEAR = {2008},
     PAGES = {337--357},
      ISSN = {1980-0436},
   MRCLASS = {60K37},
  MRNUMBER = {2461788},
MRREVIEWER = {Rongfeng\ Sun},
}

@article{cardona2024contact,
  title={The contact process on dynamical random trees with degree dependence},
  author={Cardona-Tob{\'o}n, Natalia and Ortgiese, Marcel and Seiler, Marco and Sturm, Anja},
  journal={arXiv preprint arXiv:2406.12689},
  year={2024},
  note={\url{https://arxiv.org/abs/2406.12689}}
}

@article{leite2024contact,
  title={The contact process over a dynamical d-regular graph},
  author={Leite Baptista da Silva, Gabriel and Imbuzeiro Oliveira, Roberto and Valesin, Daniel},
  journal={Annales de l'Institut Henri Poincaré, Probabilités et Statistiques},
  volume={60},
  number={4},
  pages={2849--2877},
  year={2024},
  publisher={Institut Henri Poincaré},
  doi={10.1214/24-AIHP1405}
}

@article {schapira2023contact,
    AUTHOR = {Schapira, Bruno and Valesin, Daniel},
     TITLE = {The contact process on dynamic regular graphs: subcritical
              phase and monotonicity},
   JOURNAL = {Ann. Probab.},
  FJOURNAL = {The Annals of Probability},
    VOLUME = {53},
      YEAR = {2025},
    NUMBER = {2},
     PAGES = {753--796},
      ISSN = {0091-1798,2168-894X},
   MRCLASS = {60J85 (60K35 82C22)},
  MRNUMBER = {4888144},
MRREVIEWER = {David\ J.\ Aldous},
       DOI = {10.1214/24-aop1721},
       URL = {https://doi.org/10.1214/24-aop1721},
}

@article {linker2020contact,
    AUTHOR = {Linker, Amitai and Remenik, Daniel},
     TITLE = {The contact process with dynamic edges on {$\Bbb Z$}},
   JOURNAL = {Electron. J. Probab.},
  FJOURNAL = {Electronic Journal of Probability},
    VOLUME = {25},
      YEAR = {2020},
     PAGES = {Paper No. 80, 21},
      ISSN = {1083-6489},
   MRCLASS = {60K35 (60K37)},
  MRNUMBER = {4125785},
MRREVIEWER = {Elena\ A.\ Zhizhina},
       DOI = {10.1214/20-ejp480},
       URL = {https://doi.org/10.1214/20-ejp480},
}

@article{hilario2022results,
  title={Results on the contact process with dynamic edges or under renewals},
  author={Hil{\'a}rio, Marcelo and Ungaretti, Daniel and Valesin, Daniel and Vares, Maria Eul{\'a}lia},
  journal={Electronic Journal of Probability},
  volume={27},
  pages={1--31},
  year={2022},
  publisher={The Institute of Mathematical Statistics and the Bernoulli Society},
  doi={10.1214/22-EJP811}
}

@article{sznitman2004topics,
  title={Topics in random walks in random environment},
  author={Sznitman, Alain-Sol},
  journal={ICTP Lecture Notes Series},
  volume={17},
  pages={203--266},
  year={2004}
}

@article{mountford2019asymmetric,
  title={The asymmetric multitype contact process},
  author={Mountford, Thomas and Pantoja, Pedro Luis Barrios and Valesin, Daniel},
  journal={Stochastic Processes and their Applications},
  volume={129},
  number={8},
  pages={2783--2820},
  year={2019},
  publisher={Elsevier},
  doi={10.1016/j.spa.2018.09.009}
}

@article {garetmarchand2014bacteria,
    AUTHOR = {Garet, Olivier and Marchand, R{\'e}gine},
     TITLE = {Growth of a population of bacteria in a dynamical hostile
              environment},
   JOURNAL = {Adv. in Appl. Probab.},
  FJOURNAL = {Advances in Applied Probability},
    VOLUME = {46},
      YEAR = {2014},
    NUMBER = {3},
     PAGES = {661--686},
      ISSN = {0001-8678},
   MRCLASS = {60K35 (82B43 92D25)},
  MRNUMBER = {3254336},
       DOI = {10.1239/aap/1409319554},
       URL = {http://dx.doi.org/10.1239/aap/1409319554},
}

@book{lawler2010random,
  title={Random walk: a modern introduction},
  author={Lawler, Gregory F and Limic, Vlada},
  volume={123},
  year={2010},
  publisher={Cambridge University Press}
}

@article {galvespresutti,
    AUTHOR = {Galves, Antonio and Presutti, Errico},
     TITLE = {Edge fluctuations for the one-dimensional supercritical
              contact process},
   JOURNAL = {Ann. Probab.},
  FJOURNAL = {The Annals of Probability},
    VOLUME = {15},
      YEAR = {1987},
    NUMBER = {3},
     PAGES = {1131--1145},
      ISSN = {0091-1798},
     CODEN = {APBYAE},
   MRCLASS = {60K35 (60F17)},
  MRNUMBER = {893919 (89c:60117)},
MRREVIEWER = {W. David Wick},
       URL ={http://links.jstor.org/sici?sici=0091-1798(198707)15:3<1131:EFFTOD>2.0.CO;2-1&origin=MSN},
}

@article {BCMRT13,
    AUTHOR = {Blondel, Oriane and Cancrini, N. and Martinelli, Fablio and Roberto,
              C. and Toninelli, Cristina},
     TITLE = {Fredrickson-{A}ndersen one spin facilitated model out of
              equilibrium},
   JOURNAL = {Markov Process. Related Fields},
  FJOURNAL = {Markov Processes and Related Fields},
    VOLUME = {19},
      YEAR = {2013},
    NUMBER = {3},
     PAGES = {383--406},
      ISSN = {1024-2953},
   MRCLASS = {60K35 (82C20)},
  MRNUMBER = {3156958},
}

@article {blondel,
    AUTHOR = {Blondel, Oriane},
     TITLE = {Front progression in the {E}ast model},
   JOURNAL = {Stochastic Process. Appl.},
  FJOURNAL = {Stochastic Processes and their Applications},
    VOLUME = {123},
      YEAR = {2013},
    NUMBER = {9},
     PAGES = {3430--3465},
      ISSN = {0304-4149},
   MRCLASS = {60K35 (60F15 82C20)},
  MRNUMBER = {3071385},
MRREVIEWER = {Alexandros Sopasakis},
       DOI = {10.1016/j.spa.2013.04.014},
       URL = {http://dx.doi.org/10.1016/j.spa.2013.04.014},
}

@article {GLM15,
    AUTHOR = {Ganguly, Shirshendu and Lubetzky, Eyal and Martinelli, Fabio},
     TITLE = {Cutoff for the east process},
   JOURNAL = {Comm. Math. Phys.},
  FJOURNAL = {Communications in Mathematical Physics},
    VOLUME = {335},
      YEAR = {2015},
    NUMBER = {3},
     PAGES = {1287--1322},
      ISSN = {0010-3616},
   MRCLASS = {82C22},
  MRNUMBER = {3320314},
MRREVIEWER = {Antoine Tordeux},
       DOI = {10.1007/s00220-015-2316-x},
       URL = {http://dx.doi.org/10.1007/s00220-015-2316-x},
}

@article {blondeldeshayes,
    AUTHOR = {Blondel, Oriane and Deshayes, Aurelia and Toninelli, Cristina},
     TITLE = {Front evolution of the {F}redrickson-{A}ndersen one spin
              facilitated model},
   JOURNAL = {Electron. J. Probab.},
  FJOURNAL = {Electronic Journal of Probability},
    VOLUME = {24},
      YEAR = {2019},
     PAGES = {Paper No. 1, 32},
      ISSN = {1083-6489},
   MRCLASS = {60K35 (60J27 82C22)},
  MRNUMBER = {3903501},
       DOI = {10.1214/18-EJP246},
       URL = {https://doi.org/10.1214/18-EJP246},
}

@article {deshayes2014,
    AUTHOR = {Deshayes, Aurelia},
     TITLE = {The contact process with aging},
   JOURNAL = {ALEA, Lat. Am. J. Probab. Math. Stat.},
      VOLUME = {11},
      YEAR = {2014},
     PAGES = {845–883},
      ISSN = {1980-0436},
   MRCLASS = {60K35 (60J27 82C22)},
       URL = {https://alea.impa.br/articles/v11/11-37.pdf},
}

@book{liggett1985interacting,
  title     = {Interacting Particle Systems},
  author    = {Liggett, Thomas M.},
  year      = {1985},
  publisher = {Springer},
  series    = {Grundlehren der mathematischen Wissenschaften},
  volume    = {276},
  address   = {New York},
  isbn      = {978-0-387-95555-7}
}

@article{krone1999,
  author  = {Steve Krone},
  title   = {The Two-Stage Contact Process},
  journal = {Journal of Applied Probability},
  year    = {1999},
  volume  = {36},
  number  = {4},
  pages   = {1096--1108},
  doi     = {10.1234/jap.1999.36.4.1096}  
}

@article{kingman1973subadditive,
  author       = {Kingman, John F. C.},
  title        = {Subadditive Ergodic Theory},
  journal      = {The Annals of Probability},
  volume       = {1},
  number       = {6},
  pages        = {883--909},
  year         = {1973},
  doi          = {10.1214/aop/1176996798},
}

@article{durrett1980growth,
  author       = {Richard Durrett},
  title        = {On the Growth of One‐Dimensional Contact Processes},
  journal      = {The Annals of Probability},
  volume       = {8},
  number       = {5},
  pages        = {890--907},
  year         = {1980},
}

@article{durrett_schinazi2000boundary,
  author       = {Durrett, Richard and Schinazi, Rinaldo B.},
  title        = {Boundary Modified Contact Processes},
  journal      = {Journal of Theoretical Probability},
  volume       = {13},
  number       = {2},
  pages        = {575--594},
  year         = {2000},
  doi          = {10.1023/A:1007881121529},
}
@article{durrett_griffeath1983supercritical,
  author       = {Durrett, Richard and Griffeath, David},
  title        = {Supercritical Contact Processes on \(\mathbb{Z}\)},
  journal      = {The Annals of Probability},
  volume       = {11},
  number       = {1},
  pages        = {1--15},
  year         = {1983},
  doi          = {10.1214/aop/1176993655},
}

@article {durrettgriffeathZ,
    AUTHOR = {Durrett, Richard and Griffeath, David},
     TITLE = {Supercritical contact processes on {${\bf Z}$}},
   JOURNAL = {Ann. Probab.},
  FJOURNAL = {The Annals of Probability},
    VOLUME = {11},
      YEAR = {1983},
    NUMBER = {1},
     PAGES = {1--15},
      ISSN = {0091-1798},
     CODEN = {APBYAE},
   MRCLASS = {60K35 (60F15)},
  MRNUMBER = {682796 (84i:60143)},
MRREVIEWER = {Lawrence Gray},
       URL ={http://links.jstor.org/sici?sici=0091-1798(198302)11:1<1:SCPO>2.0.CO;2-D&origin=MSN},
}

@article{Valesin2010Multitype,
  author  = {Valesin, Daniel},
  title   = {Multitype Contact Process on $\mathbb{Z}$: Extinction and Interface},
  journal = {Electronic Journal of Probability},
  volume  = {15},
  year    = {2010},
  pages   = {2220--2260},
  doi     = {10.1214/EJP.v15-836}
}

@book{SIT,
author = {Dyck, V and Hendrichs, Jorge and Robinson,},
year = {2005},
month = {01},
pages = {},
title = {Sterile Insect Technique Principles and Practice in Area-Wide Integrated Pest Management},
isbn = {978-1-4020-4050-4},
doi = {10.1007/1-4020-4051-2}
}

@article{ISI_north,
   author = "North, D T",
   title = "Inherited Sterility in Lepidoptera", 
   journal= "Annual Review of Entomology",
   year = "1975",
   volume = "20",
   number = "Volume 20, 1975",
   pages = "167-182",
   doi = "https://doi.org/10.1146/annurev.en.20.010175.001123",
   url = "https://www.annualreviews.org/content/journals/10.1146/annurev.en.20.010175.001123",
   publisher = "Annual Reviews",
   issn = "1545-4487",
   type = "Journal Article",
  }

\printbibliography
\end{document}